\documentclass[11pt]{amsart}

\usepackage{epsfig}
\usepackage{graphicx}
\usepackage{amscd}
\usepackage{amsmath}
\usepackage{amsfonts}
\usepackage{mathrsfs}
\usepackage{amssymb}
\usepackage{verbatim}
\usepackage{xcolor}
\usepackage{hyperref}

\newtheorem{theorem}{Theorem}[section]
\theoremstyle{plain}

\newtheorem{defi}[theorem]{Definition}

\newtheorem{lemma}[theorem]{Lemma}

\newtheorem{prop}[theorem]{Proposition}
\newtheorem{remark}[theorem]{Remark}

\numberwithin{equation}{section}

\def\beps{\boldsymbol{\eps}}

\def\One{{1\!\!1}}

\def\dist{{\rm dist}}

\def\Hk{{\mathcal H}}

\def\half{\frac{1}{2}}

\newcommand{\gam}{\gamma}
\newcommand{\om}{\omega}

\newcommand{\Gam}{\Gamma}
\newcommand{\sig}{\sigma}
\def\by{\vec{y}}

\def\bbe{{\bf e}}
\def\bu{{\mathbf u}}
\newcommand{\R}{{\mathbb R}}
\newcommand{\Q}{{\mathbb Q}}
\newcommand{\Z}{{\mathbb Z}}
\newcommand{\C}{{\mathbb C}}

\def\N{{\mathbb N}}

\newcommand{\Nat}{{\mathbb N}}

\def\bp{{\bf p}}

\def\wt{\widetilde}
\def\A{{\mathcal A}}

\def\Qk{{\mathcal Q}}

\def\Pk{{\mathcal P}}

\def\Sf{{\sf S}}
\def\T{{\mathbb T}}

\def\bx{{\mathbf x}}
\def\bw{{\mathbf w}}
\def\by{{\mathbf y}}
\def\bv{{\mathbf v}}
\def\bz{{\mathbf z}}

\def\be{\begin{equation}}
	\def\ee{\end{equation}}

\newcommand{\eps}{{\varepsilon}}

\def\ov{\overline}
\def\und{\underline}
\newcommand{\const}{{\rm const}}
\def\Span{{\rm Span}}

\def\a{a}

\def\ve1{\vec{1}}

\def\Ak{{\mathcal A}}

\def\Cyl{{\rm Cyl}}
\def\re{{\rm Re}}
\def\im{{\rm Im}}
\def\Tr{{\rm Tr}}

\newcommand{\abs}[1]{\lvert#1\rvert}

\providecommand{\norm}[1]{\lvert\lvert #1 \rvert\rvert}

\RequirePackage{mathtools, amsthm, graphicx, mathrsfs, dsfont}

\makeatletter
\@namedef{subjclassname@2020}{\textup{2020} Mathematics Subject Classification}
\makeatother

\begin{document}
	
	\title[Local spectral estimates for substitutions]{Local spectral estimates and quantitative \\[1.2ex] weak mixing for substitution $\Z$-actions}
	
	\author{Alexander I. Bufetov}
	\address{Alexander I. Bufetov\\ 
		Steklov  Mathematical Institute of RAS, Moscow, Russia}
	\address{Aix-Marseille Universit{\'e}, CNRS, Centrale Marseille, I2M, UMR 7373\\
		39 rue F. Joliot Curie, Marseille, France}
	\address{The University of Saint-Petersburg, Department of Mathematics and Computer Science, Saint-Petersburg, Russia}
	\email{alexander.bufetov@univ-amu.fr, bufetov@mi.ras.ru}
	\author{Juan Marshall-Maldonado}
	\address{Juan Marshall-Maldonado\\ Department of Mathematics,
		Bar-Ilan University, Ramat-Gan, Israel}
	\curraddr{Faculty of Mathematics and Computer Science, Nicolaus
		Copernicus University, ul. Chopina 12/18, 87-100, Toru\'n, Poland}
	\email{jgmarshall21@gmail.com}
	\author{Boris Solomyak }
	\address{Boris Solomyak\\ Department of Mathematics,
		Bar-Ilan University, Ramat-Gan, Israel}
	\email{bsolom3@gmail.com}
	
	\thanks{The research of J.M. and B.S. was supported by the Israel Science Foundation grant \#1647/23.
	A.B.'s research was partially supported by the BASIS Foundation under grant 24-7-1-21-1, as well as by the Ministry of Education and Science of the Russian Federation under grant 075-15-2024-631.}
	
	\begin{abstract}
		The paper investigates H\"older and log-H\"older
		regularity of  spectral measures for weakly mixing substitutions and the related question of quantitative weak mixing. It is assumed that the substitution is primitive, aperiodic, 
		and its substitution matrix is irreducible over the rationals. 
		In the case when there are no eigenvalues of the substitution matrix on the unit circle, Theorem~\ref{th-main1} says that
		a weakly mixing substitution $\Z$-action
		has uniformly log-H\"older regular spectral measures, and hence admits 
		power-logarithmic bounds for the rate of weak mixing. In the more delicate Salem substitution case, Theorem~\ref{th:Salem} says that 
		H\"older regularity holds for spectral parameters from the respective number field, but the H\"older exponent cannot be chosen uniformly.
	\end{abstract}
	
	\date{\today}


	\maketitle

	\thispagestyle{empty}
	
	\section{Introduction}
	
	The paper is devoted to the spectral theory of substitution automorphisms. Substitutions provide a rich class of measure-preserving transformations, intermediate
	between periodic and random. They also serve as models and test cases for more complicated systems, such as interval exchange transformations, finite rank systems, etc.
	Substitutions have links with combinatorics, number theory, theoretical computer science (finite state automata), harmonic analysis, and physics/material science (quasicrystals), see \cite{Siegel}.
	Although dynamical systems arising from substitutions have been studied for a long time, their spectral properties are still far from being completely understood.
	
	Here we continue the investigation of local spectral properties  for substitution automorphisms, started in \cite{BuSo14,BuSo18b} and \cite{Marshall20}.
	and obtain estimates of the decay rate of the spectral measure of a small neighborhood of a point, when the radius tends to zero. We need to assume
	that the spectral measure has no point masses, or in other words, that the system is weakly mixing. It is also known that {\em uniform} local spectral estimates (for test functions
	of zero mean) are closely related to rates of {\em quantitative weak mixing}, see e.g., \cite{Knill}.
	
	We proceed with a brief summary of our main results.
	In Theorem~\ref{th-main1} we prove that for a primitive aperiodic substitution with a substitution matrix irreducible over $\Q$ substitution matrix
	and having no eigenvalues on the unit circle, there is
	a dichotomy: either the substitution $\Z$-action has non-trivial discrete spectrum, or spectral measures of cylindrical functions of zero mean are uniformly log-H\"older regular.
	The log-H\"older estimates
	yield logarithmic decay rates of quantitative weak mixing, cf.\ Theorem~\ref{th-main2}. Uniformity here means uniform estimates for all points on the unit circle.
	
	Our second main result concerns substitutions for which the substitution matrix
	$\Sf$ is still irreducible over the rationals, but now the
	Perron-Frobenius eigenvalue of $\Sf$ is a Salem numbers $\alpha$, that is, all the remaining eigenvalues of $\Sf$ are of modulus $\le 1$, but {\em there are} eigenvalues on the
	unit circle. It is proven that at spectral parameters from the field $\Q(\alpha)$, spectral measures of cylindrical functions satisfy H\"older bounds.
	On the other hand, we show that the H\"older exponent cannot be chosen uniformly over all algebraic spectral parameters.
	
	The key difference between this paper and \cite{BuSo14,BuSo18b,Marshall20} is that here we obtain spectral estimates for substitution $\Z$-actions, whereas in those papers we 
	were mainly concerned with suspension flows over them, especially for {\em self-similar} substitution $\R$-actions. 
	It is well-known that spectral properties may change drastically under
	Kakutani equivalence, that is, time change, 
	or when passing to a suspension.
	Although our basic approach, via the
	spectral cocycle, does go back to those earlier papers, there are additional complications in the case of $\Z$-actions. 
	In particular, the criteria for weak mixing are different: Whereas for the self-similar substitution flow, absence of weak mixing is equivalent to the Perron-Frobenius eigenvalue $\alpha$ of $\Sf$ being a Pisot number, for the substitution subshift this condition is sufficient, but not necessary.

	A technical novelty is the scheme
	of approximation of toral automorphism orbits by lattice points, developed in Section 3.
	method from \cite{BuSo21}. Given a spectral parameter $\om\in [0,1]$, we consider the distances to the nearest lattice point from the orbit
	${(\Sf^{\sf T})}^n (\om \vec{1})$. It is known that $e^{2\pi i \om}$ is an eigenvalue of the substitution automorphism whenever these distances tend to zero. This is known as ``Veech criterion'' in the context of interval exchange transformations, see \cite[\S7]{veechamj}. Having these distances bounded away from zero, with some positive frequency (the usual one, or logarithmic, or something more general)
	uniformly in $\om$,
	implies a corresponding rate of quantitative weak mixing. This principle is referred to as ``Quantitative Veech criterion," see \cite[Section 5]{BuSo21} or \cite[Section 6]{AFS}. In the current paper it is manifested by equation \eqref{ineq1}, and Lemma~\ref{lem2} is the ``engine'' which drives this.
	
		Another innovation is a lower bound on the norm of a matrix product, where every matrix is a small perturbation of a fixed primitive matrix,
	see Lemma  \ref{lem:iter2}, which was used to prove lower bounds on the rates of quantitative weak mixing for Salem substitution.
	It was inspired by the {\em Avalanche Principle} of Goldstein and Schlag \cite{GS}.

	\section{Statement of results}
	
	\subsection{Background}
	The standard references for the basic facts on substitution dynamical systems are \cite{Queff,Siegel}.
	Consider an alphabet of $d\ge 2$ symbols $\Ak=\{0,\ldots,d-1\}$. Let  $\A^+$ be the set of nonempty words with letters in $\A$. 
	A {\em substitution}  is a map $\zeta:\,\A\to \A^+$, extended to 
	$\A^+$ and $\A^{\N}$ by
	concatenation. The {\em substitution space} is defined as the set of bi-infinite sequences $x\in \A^\Z$ such that any word  in $x$
	appears as a subword of $\zeta^n(\a)$ for some $\a\in \A$ and $n\in \N$. The {\em substitution dynamical system}  is the left
	shift on $\A^\Z$ restricted to $X_\zeta$, which we denote by $T$.\\
	
	The {\em substitution matrix} $\Sf=\Sf_\zeta=(\Sf(i,j))$ is the $d\times d$ matrix, such that $\Sf(i,j)$ is the number
	of symbols $i$ in $\zeta(j)$. The substitution is {\em primitive} if $\Sf_\zeta^n$ has all entries strictly positive for some $n\in \Nat$.
	It is well-known that primitive substitution $\Z$-actions are minimal and uniquely ergodic, see \cite{Queff}.
	We assume that the substitution is primitive and {\em non-periodic}, which in the primitive case is equivalent to the space $X_\zeta$ being infinite.
	The length of a word $u$ is denoted by $|u|$. The substitution $\zeta$ is said to be of {\em constant length} $q$ if $|\zeta(a)|=q$ for all $a\in \A$, otherwise, it is of {\em non-constant length}.
	
	Recall that for $f,g\in L^2(X_\zeta,\mu)$ the (complex) spectral measure $\sig_{f,g}$ is determined by the equations
	$$\widehat{\sig}_{f,g}(-k) =\int_0^1 e^{2\pi i k \om}\,d\sig_{f,g}(\om)=
	\langle f\circ T^k, g\rangle,\ \ k\in \Z,
	$$
	where $\langle \cdot,\cdot \rangle$ denotes the scalar product in $L^2$. We write $\sig_f = \sig_{f,f}$. Spectral measures ``live'' on the torus $\R/\Z$, which we identify with $[0,1)$.
	It is known that $\sig_\One = \delta_0$, where $\One$ is the constant-1 function and $\delta_0$ is the Dirac point mass at $0$. Thus, the ``interesting part'' of the spectrum is generated by functions $f$ orthogonal to constants.
	We   say that a function $f\in L^2(X_\zeta,\mu)$ is {\em cylindrical of level 0} if it depends only on $x_0$, the 0-th term of the sequence $x\in X_\zeta$. Cylindrical functions of level 0 form a $d$-dimensional vector space, with a basis
	$\{\One_{[a]}: \, a\in \Ak\}$. Denote by $\Cyl(X_\zeta)$ the space of cylindrical functions of level 0.
	
	\subsection{Log-H\"older regularity and quantitative weak mixing}
	We need to recall the criterion for $e^{2\pi i \om}$ to be an eigenvalue of the substitution 
	$\Z$-action $(X_\zeta,T,\mu)$. The criterion is due to Host \cite{Host}, with an algebraic characterization given in \cite{FMN}.
	Partial results were obtained in \cite{Liv,SolAdic}.
	For us the most convenient reference is \cite{CS1},
	which contains a criterion for eigenvalues of a 1-dimensional tiling dynamical system (suspension flow). It is known that the
	suspension $\R$-action with a constant roof function has essentially the same spectral  properties as the $\Z$-action (see \cite{Goldberg,BerSol}).
	
	A word $v\in \Ak^+$ is called a {\em return word} for $\zeta$ if $v$ starts with a letter $c$ for some $c\in \Ak$ and $vc$ is admissible for the language of $X_\zeta$. Denote by 
	$\ell(v) = [\ell(v)_j]_{j\le d}$ the {\em population vector} of $v\in \Ak^+$, where $\ell(v)_j$ is the number of letters $j$ in $v$.
	
	\begin{theorem}[{Host \cite{Host}, see also \cite{CS1}}] \label{th-host}
		Let $\zeta$ be a primitive aperiodic substitution with the incidence matrix $\Sf$. Then $e^{2\pi i \om},\ \om\in (0,1)$, 
		is an eigenvalue for the substitution $\Z$-action $(X_\zeta,T,\mu)$ if and only if for all return words $v$,
		\be \label{eq:CS1}
		\om \bigl\langle   (\Sf^{\sf T})^n{\bf 1},\ell(v)\bigr\rangle \to 0\ \mbox{\rm (mod 1)}\ \ \mbox{as}\ n\to \infty.
		\ee
	\end{theorem}
	
	It is well-known and easy to deduce from Host's Theorem that if all the eigenvalues of $\Sf$, except the Perron-Frobenius eigenvalue, are inside the
	unit circle, i.e., we are in the ``Pisot case'', then there are non-trivial eigenvalues $e^{2\pi i \om}\ne 1$. This is proved in \cite[(6.2)]{Host}, where the eigenvalues are obtained from the frequencies of letters in an arbitrary sequence $x\in X_\zeta$. Since these frequencies are irrational, we in fact obtain a dense set of eigenvalues, since eigenvalues form a subgroup of the unit circle.
	If $\om$ is an eigenvalue, then there exists a cylindrical function $f$ such that $\sig_f$ has a point mass at $\om$ and hence no meaningful estimates of the modulus continuity for $\sig_f$ at
	$\om$ are possible. Our goal is to show that under some natural assumptions on the substitution, weak mixing implies uniform log-H\"older estimates.
	
	\begin{theorem} \label{th-main1}
		Let $\zeta$ be a primitive aperiodic substitution, such that the substitution matrix $\Sf=\Sf_\zeta$ has an irreducible characteristic polynomial over $\Q$ and there are no eigenvalues of $\Sf$
		on the unit circle. Then either $(X_\zeta,T,\mu)$ has a non-trivial discrete  component of the spectrum, or there exist $C_\zeta, r_0, \gam>0$, depending only on $\zeta$, such that for 
		any cylindrical function $f$ with $\int f\,d\mu=0$, for 
		all $\om\in [0,1)$ holds
		\be \label{goal1}
		\sig_f(B_r(\om)) \le C_\zeta\|f\|^2_\infty\cdot (\log(1/r))^{-\gam},\ \ \mbox{for all}\ r\in (0,r_0).
		\ee
	\end{theorem}
	
	It is important  that in this theorem the estimate is uniform in $\om$: this implies a result on quantitative weak mixing.
	
	\begin{theorem} \label{th-main2}
		Let $\zeta$ be as in the last theorem, such that $(X_\zeta,T,\mu)$ is weakly mixing. Then there are constants $K_\zeta,\gam>0$ such that for any 
		cylindrical function $f$ with $\int f\,d\mu=0$ and any $g\in L^2(X_\zeta,\mu)$, holds
		\be \label{goal2}
		\frac{1}{N} \sum_{k=0}^{N-1} |\langle U^k f, g\rangle|^2 \le K_\zeta {\|f\|}^2_\infty \cdot {\|g\|}_2^2 \cdot (\log N)^{-\gam}.
		\ee
	\end{theorem}
	
	The derivation of \eqref{goal2} from \eqref{goal1} is well-known; see e.g., \cite{BuSo18b,Forni,AFS,Moll23}. It essentially goes back to the work of Strichartz \cite{Stri90}, see
	\cite[Theorem 3.6]{Knill}. The paper of Knill \cite{Knill} contains
	additional background on quantitative weak mixing. From another point of view, quantitative weak mixing was recently studied in \cite{Park_Park}.
	
	\begin{remark} {\em 
			
			{\bf (i)} We restrict ourselves to the case of cylindrical functions of level 0 for simplicity. It is not difficult to extend the results to the case of cylindrical functions of any level, and then
			to Lipschitz functions by approximation. Cylindrical functions of level $\ell\ge 1$ can be handled either in the framework of suspension flows (with a constant roof function) and Lip-cylindrical functions, as we do in \cite[Section 3.5 and Proposition 7.1]{BuSo18} and \cite{BuSo21}, or directly in the
			setting of substitution $\Z$-actions, as was done in \cite[Section 6]{Moll23}. Then arbitrary Lipschitz functions on the substitution spaces can be treated by approximation, as in 
			\cite[Section 9]{BuSo18} or \cite[Section 9]{Moll23}. Similar considerations are found in \cite{AFS}. \\

			{\bf (ii)} Nelson Moll \cite{Moll23} recently investigated the speed of weak mixing for the Chacon map, which is conjugate to a primitive substitution $\zeta: 0\mapsto 0012, \ 1\mapsto 12,\
			2\mapsto 012$. Note that its substitution matrix has an eigenvalue 1.
			He obtained power-logarithmic upper bounds for the spectral measures and corresponding estimates for the speed of weak mixing, for Lipschitz observables. His
			method extends to substitutions other than the one conjugate to Chacon, but seems to be limited to substitutions whose incidence matrix has an eigenvalue $\pm 1$, with
			the corresponding eigenvector having a non-trivial projection on
			the vector $\mathbf{1}$. This is, of course, very different from the Salem case. Moll \cite{Moll23} also obtains power-logarithmic lower bounds for the speed of weak mixing for which he uses more specific properties of the Chacon
			map, proved in \cite{Park_Park}.
		}
	\end{remark}

	\subsection{Substitutions of Salem type}
	The case when the substitution matrix $\Sf_\zeta$ has an eigenvalue of the unit circle is more difficult than the previous one due to the absence of hyperbolicity. Here we consider substitutions of Salem type, that is, we assume that $\Sf_\zeta$
	has an irreducible characteristic polynomial and its
	Perron-Frobenius eigenvalue is a Salem number.
	 Recall that a Salem number is a real algebraic integer greater than $1$, whose conjugates have modulus at most equal to
	$1$, with at least one having modulus equal to $1$. In fact, 
	the degree of $\alpha$ is even, one of its conjugates is $\alpha^{-1}$ (inside the unit circle), and all the rest are non-real complex conjugate pairs of modulus 1.
	In \cite{Marshall20} 
	the second-named author obtained H\"older spectral estimates for the self-similar suspension flow over an arbitrary substitution of Salem type, at the points in $\Q(\alpha)$. Here we obtain an analogous result for substitution $\Z$-actions.
	On the other hand, we show in Theorem \ref{th:lower} below that {\em uniform} H\"older bounds cannot hold either for such flows or for the $\Z$-actions.

	\begin{theorem} \label{th:Salem}
		Let $\zeta$ be an aperiodic primitive substitution of Salem type, such that its substitution matrix has an irreducible characteristic polynomial over $\Q$. Let $\alpha$ be the Perron-Frobenius eigenvalue. For every
		$\omega \in (0,1)\cap\Q(\alpha)$ there exists $C>0$ only depending on the substitution, $r_0 = r_0(\omega)>0$ and $\vartheta=\vartheta(\omega)>0$ such that
		\be \label{hoelder_salem}
		\sigma_f(B_r(\omega)) \leq C r^\vartheta, \quad \text{for all } r < r_0 \text{ and } f\in\text{\rm Cyl}(X_\zeta).
		\ee
	\end{theorem}
	\begin{remark} {\em
			The dependence of H\"older exponent $\vartheta$ on $\omega$ can be made explicit, in terms of
			$\abs{\omega}$, $\abs{\sigma_0(\om)}$ and $L\geq 1$, where $L$ is the denominator of $\omega$, 
			written in minimal form:
			\[
			\omega = \dfrac{l_0+\dots+l_{d-1}\alpha^{d-1}}{L}, \quad \gcd(l_0,\dots,l_{d-1},L) = 1.
			\]
			By a consequence of an ``approximation theorem,'' see e.g., \cite[Exercise 1; Chapter 3, Section 1]{Neukirch},  it follows that there exists $\vartheta>0$ such that \eqref{hoelder_salem} holds simultaneously for a dense set of algebraic $\om$. On the other hand, the dependence of $r_0$ on $\omega$ is  not explicit, whereas in 
			Theorem \ref{th-main1} the radius $r_0$ does not depend on $\om$ at all. In the Salem case, the existence of $r_0$ is derived from the convergence of a certain series (see equation \eqref{limit}) for which we do not   have lower bounds for the speed of convergence. In the special case of Salem numbers  of degree 4 a lower bound for $r_0$ was obtained in \cite[Proposition 6.4]{Marshall20}.}
	\end{remark}
	
	It turns out that the H\"older exponent $\vartheta$ in \eqref{hoelder_salem} cannot be chosen uniformly over all algebraic $\om \in (0,1)$ and, as a consequence, Salem substitution systems do not have the property of quantitative weak mixing with a power rate. The next theorem is most conveniently stated in terms of the local dimension of spectral measures. Recall that for a positive finite measure $\nu$ on $\R$ the lower local dimension of $\nu$ at $\om\in \R$ is defined by
	$$
	\und{d}(\nu,\om) = \liminf_{r\to 0} \frac{\log \nu(B_r(\om))}{\log r}\,.
	$$
	In the next theorem when we write ``for almost every $f\in \Cyl(X_\zeta)$'' the meaning is that $f= \sum_{j=0}^{d-1} b_j \One_{[j]}$ for Lebesgue-a.e. $(b_0,\ldots,b_{d-1})$.
	
	\begin{theorem} \label{th:lower}
		Under the assumptions of Theorem~\ref{th:Salem}, we have for almost every $f\in \Cyl(X_\zeta)$:
		\begin{enumerate}
			\item[(i)] $$\inf\bigl\{\und{d}(\sig_f,\om):\ \om \in \Z[\alpha]\cap (0,1)\bigr\} = 0;$$ 
			\item[(ii)] For almost every $f\in \Cyl(X_\zeta)$ of mean zero the following holds: for any $\vartheta>0$ and any $C>0$ there exists a  sequence $N_i\to \infty$ such that 
			$$
			\frac{1}{N_i} \sum_{k=0}^{N_i-1} \bigl| \langle U^k f, f\rangle \bigr|^2 \ge C N_i^{-\vartheta}.
			$$
		\end{enumerate}
	\end{theorem}
	
	\begin{remark} {\em
			An analogous result holds for self-similar Salem substitution flows ($\R$-actions), studied in \cite{Marshall20}. We do not know if weaker uniform (for example, log-H\"older) estimates hold in the Salem case.
		}
	\end{remark}

		\subsection{Comparison with the results on suspension flows.} In our previous work on local spectral estimates and quantitative weak mixing
		we were mainly focused on suspension flows over substitution $\Z$-actions, also known as {\em tiling flows}, see \cite{CS1} and
		references therein. Given a substitution $\zeta$, this flow is determined by a positive vector  
		$\vec{s} = (s_0,\dots,s_{d-1})^{\sf T} \in \R^d_+$, whose components are the
		{\em tile lengths}. It can also be viewed as a suspension flow over $(X_\zeta,T,\mu)$ with a piecewise-constant roof function 
		$\phi(x) = \sum_{j\in \Ak} s_j\One_{[j]}(x)$. 
		In \cite[Theorem 2.3]{CS1} Clark and
		Sadun show that $\om \in \R$ is in the point spectrum of the measure-preserving tiling flow corresponding $\vec s$ if and only if
		$$
		\om \bigl\langle  (\Sf^{\sf T})^n{\vec s},\ell(v)\bigr\rangle \to 0\ \mbox{\rm (mod 1)}\ \ \mbox{as}\ n\to \infty
		$$
		for all return words $v$. (Compare this with \eqref{eq:CS1}; note also that \cite{CS1} uses the term {\em recurrence word} for our return words.)
		Under some mild assumptions this implies that 
		if $\Sf$ has $\kappa\ge 2$ eigenvalues of modulus strictly greater than 1, then for a ``typical'' $\vec s$ the flow is weakly mixing (in fact, the set of
		exceptional $\vec s$, for which the flow has nontrivial point spectrum, has ``co-dimension'' at least $\kappa-1$).
		A natural question is whether this ``typical'' mixing can be upgraded
		to quantitative weak mixing.  In \cite{BuSo14,BuSo18b} it was shown that for {\em Lebesgue almost every} $\vec s\in \R^d_+$ the spectral measure of cylindrical functions of zero mean admits
		{\em uniform H\"older} bounds $\sig_f(B_r(\om)) \le \const\cdot r^\gam$, for some $\gam>0$, with the
		corresponding, H\"older rates of quantitative weak mixing. 
		The current paper deals with the case of $\vec s = {\bf 1}$, and we do not expect the log-H\"older bounds obtained here to be sharp. 
		
		It is a common phenomenon that results for a {\em specific} parameter are much harder than results for
		a {\em typical} parameter. (A case in point is Borel's result saying 
		that almost every number is normal, whereas few specific normal numbers are known.) This is also the scenario we deal with: although we expect 
		H\"older rates to hold for typical suspensions in the sense of Clark and Sadun, that is, for suspensions corresponding to the vector $\vec{s}$ outside a countable union of subspaces (see \cite[Theorem 2.7]{CS1}), we are only able to prove log-H\"older rates for fixed suspensions.
					
		In \cite[Theorem 5.1]{BuSo14} a result analogous to Theorem~\ref{th-main1}, that is, a log-H\"older bound, was obtained for the {\em self-similar} suspension flow, that is, when $\vec s$ is the Perron-Frobenius (PF) eigenvector of $\Sf^{\sf T}$.
		Note, however, that the conditions on the substitution are not the same. In \cite{BuSo14} it was
		required that the Perron-Frobenius eigenvalue $\theta$ of $\Sf$ has at least one Galois conjugate outside the unit circle.  
		For substitution $\Z$-actions, the presence of such a Galois conjugate is not sufficient for weak mixing, see \cite{SolAdic,FMN}.
		On the other hand, here we require irreducibility of the characteristic polynomial over $\Q$, which was not needed in \cite[Theorem 5.1]{BuSo14}.
		Our methods rely on the algebraicity of the vector ${\bf 1}$, see equation (\ref{eq31}). It would be interesting to extend our results to the more general suspensions $\vec s \in \R_{+}^d$.
		
	\begin{remark}{\em
	There is another application which highlights the differences with the previous work. The results on self-similar suspension flows from \cite{BuSo14,Marshall20} imply, in particular, regularity of spectral measures of the flows along the leaves of the stable and unstable  foliations of a pseudo-Anosov diffeomorphism (see \cite[Corollary]{BuSo14} and \cite[Section 4.1]{Bu}). Theorems \ref{th-main1}, \ref{th-main2}, \ref{th:Salem} and \ref{th:lower} instead imply analogous results for the corresponding class of self-similar interval exchange transformations (IET's), see, e.g. \cite{CG}, which is a zero measure set of the space of interval exchanges and therefore not covered by the general result from \cite{AFS}.
	
	(Incidentally, in \cite[Remark 1.8]{AFS} it is noted that the ideas of \cite[Section 7]{AFS} may be used to obtain logarithmic upper bounds for self-similar
	IET's of some special combinatorics; this is, in fact, similar to the case treated by Moll \cite{Moll23}.)
}
\end{remark}

	\subsection{Scheme of the proofs.} \label{sec:scheme}
	Local estimates for the spectral measure are deduced from growth estimates of  {\em twisted Birkhoff sums} using the following standard lemma; see e.g., \cite{Hof} and
	\cite[Lemma 3.1]{BuSo14} for a short proof.
	Let $(X,T,\mu)$ be a measure-preserving system. For $f\in L^2(X,\mu)$ let $\sig_f$ be the corresponding spectral measure on $\T = \R/\Z\cong [0,1)$.
	For $x\in X$ and $f\in L^2(X,\mu)$ let
	$$
	S_N^x(f,\om) = \sum_{n=0}^{N-1}e^{-2\pi i n\om} f(T^n x)\ \ \ \mbox{and}\ \ \ G_N(f,\om) = N^{-1} \int_X |S_N^x(f,\om)|^2\,d\mu(x).
	$$
	
	\begin{lemma} \label{lem-easy1}
		For all $\om\in[0,1)$ and $r\in (0,\half]$ we have
		\be \label{eq-estim1}
		\sig_f(B(\om,r))\le \frac{\pi^2}{4N} G_N(f,\om),\ \ \ \mbox{with}\ \ N = \lfloor (2r)^{-1}\rfloor.
		\ee
	\end{lemma}
	
	In the setting of Theorem~\ref{th-main1} we actually obtain estimates of twisted Birkhoff sums for the substitution system $(X_\zeta,T,\mu)$ that are uniform in $x\in X_\zeta$:
	$$
	|S_N^x(f,\om)|\le C_f(\om)\cdot N(\log N)^{-\gam},\ \ \om\ne 0,\ \ N\ge N_0(\om),
	$$
	and then the lemma yields
	\begin{equation}\label{estimate}
		\sig_f(B(\om,r)) \le  C'_f(\om) (\log(1/r))^{-2\gam}\ \ \om\ne 0,\ \ r\le r_0(\om).
	\end{equation}
	
	For $\om=0$  stronger, H\"older bounds hold, which follow from well-known estimates of the usual Birkhoff sums, essentially due to \cite{Adam}, and then \eqref{goal1} is
	obtained by ``gluing'' it with the one in \eqref{estimate}, which requires explicit control of $r_0(\om)$ and $C'_f(\om)$.
	
	Growth estimates for twisted Birkhoff sums are proved by a variant of the method applied in \cite[Section 5]{BuSo14} to self-similar flows. It also uses considerations of
	Diophantine nature, but is more geometric.

	In the setting of Theorem~\ref{th:Salem},
	given $\om\in (0,1)\cap\Q(\alpha)$, we show uniform in $x$
	bounds
	\be \label{bsum2}
	|S_N^x(f,\om)|\le \wt C_f(\om)\cdot N^{\wt\vartheta},\ \ \wt\vartheta\in (0,1),\ \ n\ge N_0(\om),
	\ee
	which imply, via Lemma~\ref{lem-easy1}, the bound
	$$
	\sig_f(B(\om,r)) \le  \wt C'_f(\om)\cdot r^{2(1-\wt\vartheta)}, \ \ r\le r_0(\om).
	$$
	The estimate \eqref{bsum2} is proven by a variant of the method used in \cite{Marshall20} for self-similar Salem type flows; however, the extension is not automatic.
	Additional effort is needed to make the dependence of $\wt\vartheta$ on $\om$  \textit{effective}: we say that a constant depending on some parameters is effective if this dependence can be made explicit.
	
	The proof of Theorem~\ref{th:lower}(i) uses several ingredients: a choice of $\om \in \Z[\alpha]$ with some specific Diophantine properties, a lemma on the lower
	bound of the norm of a matrix product, when every matrix is a small perturbation of a fixed primitive matrix, and a result from \cite{BuSo20}, relating the lower local dimension of
	the spectral measure to the pointwise upper Lyapunov exponent of the spectral cocycle. Part (ii) follows from part (i) immediately, by \cite[Theorem 3.7]{Knill},
	which is a general version of Last's theorem
	\cite[Theorem 3.1]{Last96}.


	\section{Approximation of toral automorphisms orbits by lattice points} \label{sec:approx}
	
	Let $A\in GL_d(\Z)$, $\om\in (0,1)$, ${\bf 1} = (1,\ldots,1)^{\sf T}$, and $L< \Z^d$ a lattice, such that $AL\subset L$. Consider 
	$$
	A^n \om{\bf 1} = \bp_n + \beps_n,\ \ n\ge 0.
	$$
	where $\bp_n\in L$ is the nearest lattice point to $A^n \om{\bf 1}$ in the norm ${\|\cdot\|}_\infty$. Below we will use $\ell^\infty$ metric and matrix norm, unless stated otherwise.
	Clearly, $\bp_n$ and $\beps_n$ depend on $\om$; we keep this in mind, but do not indicate it explicitly in our notation.
	
	Let $a_L>0$ be the minimal distance between distinct points in $L$ and let $b_L>0$ be such that the union of balls of radius $b_L$ with centers in $L$ covers $\R^d$. Thus we have
	$$
	{\|\beps_n\|} \le b_L,\ \ n\in \N.
	$$
	Observe that $b_L\ge \half$, with equality for $L = \Z^d$.
	Define
	$$
	\bz_0 := -\bp_0,\ \ \bz_n:= A\bp_{n-1} - \bp_{n},\ \ n\ge 1.
	$$
	Note that $A(\bp_{n-1} + \beps_{n-1}) = \bp_n + \beps_n$, hence
	\be \label{eq1}
	\bz_n= A\bp_{n-1} - \bp_{n} = \beps_{n} - A \beps_{n-1} \in L,\ \ n\ge 1.
	\ee
	Thus we immediately obtain:
	\begin{lemma} \label{lem1}
		{\rm (i)} We have
		$$
		\| \bz_n\|  \le (1 + \|A\|)\cdot b_L \ \ \mbox{for all}\ n\ge 0,
		$$
		hence there exists a finite set $F\subset L$ such that $\bz_n\in F$ for $n\ge 0$.
		
		\smallskip
		
		{\rm (ii)} If $$\max\{\|\beps_{n-1}\|, \|\beps_{n}\|\} < \frac{a_L}{4\|A\|}\,,$$ then $\bz_n = \bf{0}$ and $\beps_{n} = A\beps_{n-1}$.
	\end{lemma}
	
	In part (i) the inequality holds for $n=0$ since $\|\bz_0\| = |\om| < 1< 2b_L$.
	
	\medskip
	
	Now let us assume that $A$ has an irreducible characteristic polynomial  over $\Q$; denote it by $p(x)$. Then the eigenvalues of $A$ are all simple; they are
	algebraic integers conjugate to each other. 
	We also assume that $A$ is primitive, so that
	there is a dominant PF eigenvalue $\theta_1>1$. Enumerate the eigenvalues by magnitude:
	$$
	\theta_1 > |\theta_2|\ge\ldots 
	$$
	We assume that there are $\kappa\ge 2$ eigenvalues strictly greater than 1 in absolute value. 
	We can choose a basis
	$\{\bbe^*_j\}_{j\le d}$ consisting of eigenvectors for the transpose matrix $A^{\sf T}$, in such a way that 
	$\bbe_j^*$ have their components in the ring $\Z[\ov\theta_j]$; that is,
	$\bbe^*_j \in (\Z[\ov\theta_j])^d$. In fact,
	fixing  the 1st component of the PF eigenvector for $A^{\sf T}$ to be equal to 1, we can find the other components by Cramer's rule. Then
	the other components are going to be elements of $\Q(\theta_1)$. Multiplying this eigenvector by a scalar, we can get rid of the denominators and 
	obtain $\bbe^*_1$ with components in $\Z[\theta_1]$. The Galois group of an irreducible integer polynomial is transitive (see \cite[Section 14.6]{Dummit}) ; let $\tau_j$ be an automorphism of the splitting field of $p(x)$ mapping
	$\theta_1$ to $\theta_j$. Then $\ov\tau_j(\theta_1) = \ov\theta_j$ and we can define
	$$
	 \bbe^*_j:=\ov\tau_j (\bbe^*_1).
	$$
to obtain the desired basis.
	
	Thus, for any $\bz \in \Z^d$ there exists a polynomial $\Pk_\bz \in \Z[x]$ of degree $\le d-1$, such that
	\be \label{def:Pz}
	\langle \bz, \bbe^*_j \rangle = \Pk_\bz(\theta_j),\ \ j\le d.
	\ee
	Notice that $\Pk_\bz$ does not depend on $j$, since
	\be \label{eq:ind1}
	\tau_j(\langle \bz,\bbe_1^*\rangle) = \langle \bz,\bbe^*_j\rangle.
	\ee

	We can write for $n\ge 1$:
	\begin{eqnarray*}
		\beps_{n} = \bz_n + A \beps_{n-1} & = & \bz_n + A(\bz_{n-1} + A\beps_{n-2})\\
		& = & \bz_n + A\bz_{n-1} + A^2\beps_{n-2} = \ldots \\
		& = & \bz_n + A\bz_{n-1} + \cdots + A^{n-1}\bz_{1} + A^{n} \beps_0.
	\end{eqnarray*}
	Denote $b_n^{(j)}:= \langle \bz_n, \bbe^*_j\rangle$ for $n\ge 0$; then we obtain for all $j\le d$:
	\be \label{eq2}
	\langle \beps_{n}, \bbe_j^*\rangle  = b_n^{(j)} +  \theta_j b_{n-1}^{(j)} + \cdots +  \theta_j^{n-1} b_1^{(j)} + \theta_j^{n} \langle\beps_0,\bbe^*_j\rangle.
	\ee
	Similarly, for $n,k\in \N$ and $j\le d$ we have
	$$
	\langle \beps_{n+k}, \bbe_j^*\rangle = b_{n+k}^{(j)} +  \theta_j b_{n+k-1}^{(j)} + \cdots +  \theta_j^{k-1} b_{n+1}^{(j)} + \theta_j^{k}\langle \beps_{n}, \bbe_j^*\rangle.
	$$
	Observe that
	$$
	|\langle \beps_{n+k}, \bbe_j^*\rangle| \le {\|\bbe_j^*\|}_1\cdot b_L,
	$$
	hence we obtain 
	
	\begin{lemma} \label{lem2}
		At least one of the following possibilities hold:
		\begin{enumerate}
		\item[(i)]
		 not all of $\bz_{n+1},\ldots,\bz_{n+k}$ are equal to zero;
\item[(ii)]
		$$
		|\langle \beps_{n}, \bbe_j^*\rangle| \le |\theta_j|^{-k} {\|\bbe_j^*\|}_1\cdot b_L.
		$$
		\end{enumerate}
	\end{lemma}
	
	\medskip
	
	Since $\beps_0 = \om{\bf 1} - \bp_0=\om{\bf 1} + \bz_0$, equation \eqref{eq2} implies for all $j\le d$:
	\be \label{eq3}
	\om \langle {\bf 1}, \bbe_j^*\rangle = -b_0^{(j)} - \frac{b_1^{(j)}}{\theta_j} - \frac{b_2^{(j)}}{\theta_j^2} - \ldots - \frac{b_n^{(j)}}{\theta_j^{n}} + 
	\frac{\langle \beps_{n}, \bbe_j^*\rangle}{\theta_j^{n}} =: Q_n(1/\theta_j) + \frac{\langle \beps_{n}, \bbe_j^*\rangle}{\theta_j^{n}},
	\ee
	for some polynomial $Q_n \in \Z[X]$. 
	Let us define $Q_n$ precisely. Observe that $\Z[\theta_1]\subseteq \Z[1/\theta_1]$, since $\theta_1$ is an algebraic integer. It follows that for any polynomial
	$p\in \Z[X]$ of degree $\le d-1$ there exists a polynomial $\wt p\in \Z[X]$ of degree $\le d-1$ , such that $\wt p(1/\theta_1) = p(\theta_1)$, and then
	$$
	\wt p(1/\theta_j) = p(\theta_j),\ \ j\le d.
	$$
 We apply this operation to $\Pk_\bz$ for $\bz\in \Z^d$ to obtain a polynomial $\wt \Pk_\bz$ and define
	$$
	Q_n(x) = \wt \Pk_{\bz_0}(x) + \wt \Pk_{\bz_1}(x) x + \cdots + \wt \Pk_{\bz_n}(x) x^n.
	$$
	 With this definition equation \eqref{eq3} holds for $j\le d$.
	Observe that $\deg(Q_n)\le n+d-1$, and since $\bz_j\in F$ for all $j\in \N$, it is easy to see that all the coefficients of $Q_n$ belong to a finite set, independent of $n$.

	Note that $\langle {\bf 1}, \bbe_j^*\rangle\ne 0$, since the matrix is irreducible over $\Q$ and hence components of eigenvectors are rationally independent. Thus, for all $j\le d$,
	\be \label{omega}
	\om = \frac{Q_n(1/\theta_j)}{\langle {\bf 1}, \bbe_j^*\rangle} + \frac{\langle \beps_{n}, \bbe_j^*\rangle}{\theta_j^{n}\langle {\bf 1}, \bbe_j^*\rangle} =: P_n(1/\theta_j) + \delta_j^{(n)},
	\ee
	for some polynomial $P_n\in \Q[X]$, and
	\be \label{delta}
	\delta_j^{(n)} = \frac{\langle \beps_{n}, \bbe_j^*\rangle}{\theta_j^{n}\langle {\bf 1}, \bbe_j^*\rangle}. 
	\ee
	In fact, $\langle {\bf 1}, \bbe_j^*\rangle  \in \Z[\theta_j]$, hence $1/\langle {\bf 1}, \bbe_j^*\rangle = \Qk(1/\theta_j)$ for some polynomial $\Qk\in \Q[X]$, since non-zero elements of $\Q[\theta_j]$ are invertible.
	The polynomial $P_n$ is obtained as a product of $Q_n\in \Z[x]$ and  a fixed rational polynomial $\Qk$, hence there exists $q\in \N$, depending only on $A$,
	such that
	\be \label{def:q}
	qP_n \in \Z[X],\ \ n \in \N,
	\ee
	that is, the coefficients of $P_n$ are in $q^{-1}\Z$.
	Moreover, the polynomials $P_n$ have coefficients from a bounded set 
	independent of $n$, and $\deg(P_n) \le n+2(d-1)$. 
	
	Until now, we worked with an arbitrary $j\le d$, but from now on we restrict ourselves to $j\le \kappa$, so that $\delta_j^{(n)} \to 0$ as $n\to \infty$.
	Equation \eqref{omega} implies  that for $2\le j \le \kappa$,
	\be \label{eq31}
	P_n(1/\theta_j) - P_n(1/\theta_1) = \delta_1^{(n)} - \delta_j^{(n)}.
	\ee
	The following is a standard fact; for a proof see \cite[Lemma 5.10]{Hochman_14}.
	
	\begin{lemma}\label{hochman}
		Let $E$ be a finite set of algebraic numbers over $\Q$. Then there exists $s\in (0,1)$ such that for any polynomial expression $x$ of degree $n$ in the elements of $E$, either $x=0$ or
		$|x| \ge s^n$.
	\end{lemma}
	
	\begin{remark}\label{rem:expicit} {\em 
			The constant $s$ can be made explicit: any integer polynomial expression on a finite set of algebraic numbers might be turned into an integer polynomial evaluated at a primitive element of the extension field generated by the algebraic numbers, and a primitive element can always be computed (see for example, \cite{Yokoyama}). In summary, in our case the constant will depend only on the substitution.}
	\end{remark}

	Thus we obtain from \eqref{eq31} that 
	\be \label{fact1}
	\mbox{for all $j=2,\ldots,\kappa$, \ either}\ \ \delta_1^{(n)} = \delta_j^{(n)}, \ \mbox{or}\  \ |\delta_1^{(n)} - \delta_j^{(n)}|\ge s^n, 
	\ee
	where $s$ depends only on the matrix $A$. 
	Fix $K\in \N$ (which depends only on $A$) such that
	\be \label{ineq0}
	|\theta_j|^{-(K+1)n + 2} < \frac{s^n}{{\|\bbe_j^*\|}_1\cdot b_L}\ \ \mbox{for all}\ \ j\le \kappa\ \ \mbox{and for all}\ \ n\ge 1.
	\ee

	\begin{lemma} \label{lem-main}
		For any $n\in \N$,  at least one of the possibilities hold:
		\begin{enumerate}
		\item[(i)] $\delta_1^{(n)} = \delta_j^{(n)}$ for all $j\le \kappa$;
		\item[(ii)]
		$
		\max\{\|\bz_{n+1}\|,\ldots,\|\bz_{n+Kn-2}\|\} >0.
		$
		\end{enumerate}
		
	\end{lemma}
	
	\begin{proof}
		Suppose that there exists $1<j\le \kappa$ such that 
		$\delta_1^{(n)} \ne \delta_j^{(n)}$. Then $\max\{|\delta_1^{(n)}|, |\delta_j^{(n)}|\}\ge  s^n$ by \eqref{fact1}. In view of \eqref{delta} and \eqref{ineq0}, we obtain 
		$$
		|\langle \beps_{n}, \bbe_j^*\rangle| \ge  s^n |\theta_j|^{n} > |\theta_j|^{-Kn+2}\cdot  {\|\bbe_j^*\|}_1\cdot b_L.
		$$
		Then Lemma~\ref{lem2} implies that one of the numbers $\bz_{n+1},\ldots,\bz_{n+Kn-2}$ is non-zero.
	\end{proof}

	
	\section{Proof of Theorem~\ref{th-main1}} \label{ProofSec}
	
	\begin{proof}
		It is known since \cite{Host} that if $|\theta_2|<1$ (the ``Pisot case''), then the discrete component of the spectrum is non-trivial, so we can assume that $\kappa\ge 2$.
		
		Consider the Abelian group $\Gam$ (a subgroup of $\Z^d)$ generated by the population vectors $\ell(v)$, where $v$ are return words for $X_\zeta$.
		We have $\Sf\Gam\subset \Gam$, where $\Sf = \Sf_\zeta$ is the substitution matrix,
		hence the real span of $\Gam$ is a rational $\Sf$-invariant subspace. Since $\Sf$ is irreducible, this must be all of $\R^d$.
		It follows that $\Gam$ has full rank, hence it is a lattice.

		Consider the {\em twisted Birkhoff sum}
		$$
		S_N^x(f,\om) = \sum_{n=0}^{N-1} e^{-2\pi i n \om} f(T^n x),
		$$
		where $T$ is the left shift on $X_\zeta$ and $x\in X_\zeta$. Recall that $v$ is called a {\em good return word} for the substitution $\zeta$ if $v$ starts with some letter $c\in \Ak$ and
		$vc$ occurs as a subword in $\zeta(b)$ for every $b\in \Ak$. The set of good return words is obviously finite; denote it by $GR(\zeta)$.
		Observe that the set $GR(\zeta)$ is nonempty once one passes to a sufficiently large power substitution, which by the primitivity assumption, does not change the space $X_\zeta$.
		
		The following is an immediate corollary of \cite[Proposition 3.5]{BuSo14}. Recall that $\theta = \theta_1$.
		
		\begin{prop}[{\cite{BuSo14}}] \label{prop:BuSo}
			Let $\zeta$ be a primitive aperiodic substitution on $\Ak$. There exist $C_2,c_3,C_4>0$ depending only on $\zeta$ such that for any cylindrical function $f:X_\zeta\to \C$ of level zero,
			any $x\in X_\zeta$, any $\om\in (0,1)$, and any $N\in \N$, holds
			$$
			\bigl|S_N^x(f,\om)\bigr| \le C_4 {\|f\|}_\infty\cdot N\cdot \!\!\!\!\!\prod_{n=0}^{\lfloor \log_{\theta} N - C_2\rfloor} \bigl(1 - c_3\cdot \max_{v\in GR(\zeta)} {\bigl\|\om|\zeta^n(v)|\bigr\|}_{\R/\Z}^2\bigr).
			$$
		\end{prop}
		Here $\|x\|_{\R/\Z}$ denotes the distance from $x\in \R$ to the nearest integer.
		Passing to a power of $\zeta$, we can make sure that all return words of less than certain length are good, and then choose 
		$v_1,\ldots,v_k\in GR(\zeta)$ such that $\Z[\ell(v_1),\ldots,\ell(v_k)] =\Gam$. Note, however, that $\{\ell(v_1),\ldots,\ell(v_k)\}$ is not necessarily a basis for the lattice; we may have
		$k>d$.
		Observe that
		$$
		|\zeta^n(v)| = \langle {\bf 1}, \Sf^n \ell(v) \rangle = \langle (\Sf^{\sf T})^n {\bf 1},\ell(v)\rangle,
		$$
		and an elementary lemma implies that
		\be \label{eq-elem1}
		\max_{j \le k} {\bigl\|\om|\zeta^n(v_j)|\bigr\|}_{\R/\Z} = \max_{j \le k} \bigl|\langle   (\Sf^{\sf T})^n \om{\bf 1},\ell(v_j)\rangle\bigr| \asymp {\bigl\|(\Sf^{\sf T})^n \om{\bf 1}\bigr\|}_{\R^d/\Gam^*},
		\ee
		where 
		$$
		\Gam^*:= \{\bv\in \R^d:\ \langle \bv,\bz\rangle\in \Z\ \ \mbox{for all}\ \bz\in \Gam\}
		$$
		is the dual lattice for $\Gam$. 
		The constants implied in $\asymp$ may depend on $\zeta$ and on $v_j$, but not on $\om$ and $n$.
		\begin{lemma} \label{lem:lattice}
			Let $\Gam = \Z[\bx_1,\ldots,\bx_k]$ be a full-rank lattice in $\R^d$, with $k\ge d$, and let $\Gam^*$ be the dual lattice. Then
			$$
			\|\bu\|_{\R^d/\Gam^*} \asymp \max_{j\le k} \big\|\langle \bu, \bx_j\rangle\bigr\|_{\R/\Z},
			$$
			where the implied constants are allowed to depend on $\bx_j$ (and not just on $\Gam$).
		\end{lemma}
		
		\begin{proof}
			By definition, $\|\bu\|_{\R^d/\Gam^*} =\min\{\|\bu - \bv\|:\ \bv\in \Gam^*\}$. For any $\bv\in \Gam^*$ and $j \le k$, we have $\langle \bv, \bx_j \rangle\in \Z$, hence
			$$
			\big\|\langle \bu, \bx_j \rangle\bigr\|_{\R/\Z} \le |\langle \bu-\bv, \bx_j\rangle| \le  \|\bu-\bv\|\cdot {\|\bx_j\|}_1.
			$$
			Thus,
			$$
			\max_{j \le k} \big\|\langle \bu, \bx_j \rangle\bigr\|_{\R/\Z} \le \bigl(\max_{j\le k} {\|\bx_j\|}_1\bigr)\cdot \|\bu\|_{\R^d/\Gam^*}.
			$$
			
			In order to obtain an estimate in the other direction, choose a subset of $\{\bx_1,\ldots,\bx_m\}$, which is linearly independent and whose linear $\R$-span equals $\R^d$. Without loss of generality
			we can assume that $\Span_\R\{\bx_1,\ldots,\bx_d\} = \R^d$. 
			This implies that $\Gam \supseteq \Z[\bx_1,\ldots,\bx_d]$, hence $\Gam^* \subseteq \bigl(\Z[\bx_1,\ldots,\bx_d]\bigr)^*$.
			For each $j\le d$ let $b_j\in \Z$ be the nearest integer to $\langle \bu,\bx_j\rangle$,which means, by definition, that
			$\big\|\langle \bu, \bx_j \rangle\bigr\|_{\R/\Z} = |\langle \bu,\bx_j\rangle-b_j|$. Let $\{\by_k\}_{k\le d}$ be the dual basis of $\R^d$ for $\{\bx_j\}_{j\le d}$, 
			that is, $\langle \by_k,\bx_j\rangle = \delta_{kj}$. 
			Then $\bv = \sum_{j=1}^k b_j \by_j \in \bigl(\Z[\bx_1,\ldots,\bx_d]\bigr)^*$, hence $\|\bu\|_{\R^d/\Gam^*}\le \|\bu - \bv\|$. The set $\{\by_1,\ldots,\by_d\}$ is also an $\R$-basis for $\R^d$, hence there
			exist $c_j\in \R$ such that $\bu = \sum_{j=1}^d c_j \by_j$. We have
			$$
			\|\bu\|_{\R^d/\Gam^*}\le \|\bu - \bv\| \le d\cdot \bigl(\max_{j\le d} |c_j - b_j|\bigr) \cdot \max_{j\le d} \|\by_j\|.
			$$
			But $c_j =\langle \bu,\bx_j \rangle$ and $b_j\in \Z$, so $|c_j - b_j| \le \big\|\langle \bu,\bx_j\rangle\bigr\|_{\R/\Z}$, and the proof is finished.
		\end{proof}

		Thus, by \eqref{eq-elem1} and Proposition~\ref{prop:BuSo},
		\be \label{ineq1}
		\bigl|S_N^x(f,\om)\bigr| \le C_4 {\|f\|}_\infty\cdot N\cdot  \!\!\!\!\!\prod_{n=0}^{\lfloor \log_{\theta} N - C_2\rfloor} \bigl(1 - \wt c_3\cdot {\bigl\| (\Sf^{\sf T})^n \om{\bf 1}\bigr\|}_{\R^d/\Gam^*}^2\bigr),
		\ee
		where $\wt c_3>0$ also depends only on $\zeta$. \\

		We will run the argument from Section~\ref{sec:approx} with $\om$, $A = \Sf^{\sf T}$, and $L = \Gam^*$, that is,
		$$
		(\Sf^{\sf T})^n \om {\bf 1} = \bp_n + \beps_n,
		$$
		and $\|\beps_n\| = {\| (\Sf^{\sf T})^n \om{\bf 1}\|}_{\R^d/\Gam^*}$ by the definition of the norm.\\
		
		In order to motivate what happens next, let us describe informally the rest of the proof. If for all $n\in \N$ sufficiently large there exists
		$j\le \kappa$, with $\delta_1^{(n)} \ne \delta_j^{(n)}$, then Lemma~\ref{lem-main}, together with Lemma~\ref{lem1}(ii) and Proposition~\ref{prop:BuSo} in the form \eqref{ineq1}, easily imply
		$$
		|S_N^x(f,\om)| \le C(\om)\cdot \|f\|_\infty\cdot N(\log N)^{-\gam},\ \ N\ge N_0(\om),
		$$
		for some $\gam>0$.
		In order to control $C(\om)$ and $N_0(\om)$, we need to specify the meaning of ``sufficiently large''. It is clear that we have to quantify the distance of $\om$ to $0$, since
		if $\om$ is very small, then  $P_n\equiv 0$ for small $n$, and then trivially $\delta_j^{(n)}=\om$ for all $j\le d$. However, for technical reasons we will also need to control the distance
		of $\om$ to $q^{-1}\Z$, where $q$ is from \eqref{def:q}. Indeed, if $\delta_j^{(n)}=\om$ for all $j\le \kappa$, for some $n\in \N$, then $P_n(1/\theta_j) = P_n(1/\theta_1),\ j\le \kappa$.
		Then a simple argument going back to \cite{SolAdic}, which was motivated by \cite{Liv,Liv2},
		shows that $\exp[2\pi i qP_n(1/\theta_1)]$ is an eigenvalue for the substitution $\Z$-action. However, we need this eigenvalue to 
		be nontrivial, that is, $P_n(1/\theta_1)\notin q^{-1}\Z$. Since $\delta_1^{(n)}\to 0$ as $n\to \infty$, the equation \eqref{omega} shows the need to quantify the distance of 
		$\om$ to $q^{-1}\Z$.
		Next we realize this scheme carefully.

		\medskip
		
		{\bf Case 1.} Suppose first that $\om \not\in q^{-1}\Z$
		and let $n_0\in \N$ be such that
		\be \label{supp1}
		\dist(\om,q^{-1}\Z) > c_1\theta_1^{-n_0},\ \ \mbox{where}\ \ c_1 = \frac{b_L {\|\bbe_1^*\|}_1}{|\langle {\bf 1}, \bbe_1^*\rangle|}.
		\ee
		
		{\bf Case 1a}.
		Suppose further that for all $n\ge n_0$ there exists $j\in \{2,\ldots,\kappa\}$ such that $\delta_1^{(n)} \ne \delta_j^{(n)}$.
		Then by Lemma~\ref{lem-main} for all $n\ge n_0$ we have $\max\{|\bz_{n+1}|,\ldots, |\bz_{n+Kn-2}|\}>0$. By Lemma~\ref{lem1}(ii), it follows that
		\be \label{ineq15}
		\max\{\|\beps_{n+1}\|,\ldots, \|\beps_{n+Kn-1}\|\} \ge \rho:= \frac{a_L}{4\|\Sf\|}\  \ \ \mbox{for all}\ n\ge n_0.
		\ee
		Thus \eqref{ineq1} implies
		\be \label{ineq2}
		N \ge \theta^{C_2 + n_0(K+1)^\ell} \implies  \bigl|S_N^x(f,\om)\bigr| \le C_4 {\|f\|}_\infty\cdot N\cdot  \bigl(1 - \wt c_3\rho^2\bigr)^\ell,\ \ \mbox{for all}\ \ell\in \N.
		\ee
		Together with \eqref{supp1} this yields by simple algebraic manipulations 
		that there exist $N_0\in \N$ and $\wt C_4$, depending only on $\zeta$, such that
		\be \label{ineq3}
		\bigl|S_N^x(f,\om)\bigr| \le \wt C_4  {\|f\|}_\infty\cdot N(\log N)^{-\gam} \cdot \log\left(\frac{1}{\dist(\om,q^{-1}\Z)}\right)^\gam,\ \ N\ge N_0,
		\ee
		where
		\be \label{gamma}
		\gam = \frac{-\log(1 - \wt c_3\rho^2\bigr)}{K+1}\,.
		\ee
		We should emphasize that the formula for $\gamma$ can be made explicit, 
		depending only on $\zeta$, in view of \eqref{ineq0}, \eqref{ineq1}, and \eqref{ineq15}.\\
		
		{\bf Case 1b.} The remaining case is that there exists $n_1\ge n_0$, where $n_0$ is from \eqref{supp1}, such that $\delta_j^{(n_1)} = \delta_1^{(n_1)}$ for all $j\le \kappa$. 
		Then
		$$
		|\delta_1^{(n_1)}|\le \frac{\|\beps_{n_1+1}\|\cdot  {\|\bbe_1^*\|}_1}{\theta_1^{n_1} \|\langle {\bf 1}, \bbe_1^*\rangle|} \le \frac{b_L {\|\bbe_1^*\|}_1}{\theta_1^{n_1} |\langle {\bf 1}, \bbe_1^*\rangle|} = c_1\theta_1^{-n_1}
		$$
		by  \eqref{omega}, and it follows from \eqref{supp1} and \eqref{omega} that $P_{n_1}(1/\theta_1)\notin q^{-1}\Z$.
		Recall that by definition, see \eqref{omega}, $qP_n \in \Z[X]$ for some $q\in \N$, where $q$ is independent of $n$. Let 
	\be \label{eq:eigenvalue}
		\alpha:= qP_{n_1}(1/\theta_j),\ \ j\le \kappa.
      \ee
		We have that $\alpha\not\in
		\Z$, and we will show that  $e^{2\pi i \alpha}$ is an eigenvalue for the system $(X_\zeta,T,\mu)$, similarly to \cite{SolAdic}. By Theorem~\ref{th-host}, it suffices to show that $\alpha L^v_n \to 0$ (mod 1), where $L^v_n = \langle {\bf 1}, \Sf^n \ell(v)\rangle$, for all return words $v$. 
		The sequence $\{L_n^v\}_{n\ge 0}$ satisfies the integer recurrence equation, corresponding to the characteristic polynomial $p(x)$ of
		the matrix $\Sf$. We will prove that $\alpha L_n\to 0$ (mod 1) for every such recurrent sequence. In fact, suppose  that $\{L_n\}_{n\ge 0}$
		is an integer sequence such that
		$$
		L_{n+d} = \sum_{j=0}^{d-1} a_j L_{n+j}\ \ \mbox{for}\ n\ge 0,\ \ \mbox{where}\ \ p(x) = x^d-a_{d-1} x^{d-1} - \cdots - a_0.
		$$
		
		As it is well-known, there exist $c_j\in \C$ such that
		$$
		L_n = \sum_{j=1}^d c_j \theta_j^n,\ \ n\ge 0.
		$$
		Since $qP_{n_1}$ is a polynomial with integer coefficients, we can write
		$q  P_{n_1}(x) = b_0 + b_1 x + \cdots + b_\ell x^\ell$ for some $\ell\in \N$ and $b_i \in \Z$. Consider the following sequence of integers, for $n> \ell$:
		$$
		K_n  :=  b_0 L_{n} + b_1 L_{n-1} + \cdots + b_\ell L_{n-\ell}.
		$$
		We have
		\begin{eqnarray*}
			K_n & = & \sum_{i=0}^\ell b_i \sum_{j=1}^d c_j \theta_j^{n-i} = \sum_{j=1}^d c_j \theta_j^n \sum_{i=0}^\ell b_i \theta_j^{-i}\\
			& = & \sum_{j=1}^d c_j \theta_j^n \cdot (q P_{n_1})(1/\theta_j) \\
			& = & \sum_{j=1}^\kappa c_j \theta_j^n \cdot \alpha + \sum_{j=\kappa+1}^d c_j \theta_j^n \cdot (q P_{n_1})(1/\theta_j) \\
			& = & \alpha L_n + \sum_{j=\kappa+1}^d c_j \theta_j^n \cdot \bigl[(q P_{n_1})(1/\theta_j) -\alpha\bigr].
		\end{eqnarray*}
		Since $\max_{j\ge \kappa+1} |\theta_j|<1$, it follows that $\alpha L_n - K_n\to 0$ as $n\to \infty$, and hence $\alpha L_n \to 0$ (mod 1), as desired.
		In the last displayed line we used \eqref{eq:eigenvalue}.

\medskip
		
		{\bf Case 2.} Now suppose that $\om \in q^{-1}\Z$. Recall that $\om\in (0,1)$.
		Let $n_0\in \N$ be such that
		\be \label{supp2}
		q^{-1} > c_1\theta_1^{-n_0},\ \ \mbox{where}\ \ c_1 = \frac{b_L {\|\bbe_1^*\|}_1}{|\langle {\bf 1}, \bbe_1^*\rangle|}.
		\ee

		{\bf Case 2a.}
		Suppose that for all $n\ge n_0$ there exists $j\in \{2,\ldots,\kappa\}$ such that $\delta_1^{(n)} \ne \delta_j^{(n)}$. Then we proceed exactly as in Case 1a to obtain
		\be \label{ineq30}
		\bigl|S_N^x(f,\om)\bigr| \le C_4'  {\|f\|}_\infty\cdot N(\log N)^{-\gam},\ \ N\ge N_0,
		\ee
		for some other constant $C_4'$
		with $\gamma$ from \eqref{gamma}
		
		\medskip
		
		{\bf Case 2b.} 
		The remaining case is that there exists $n\ge n_0$, where $n_0$ is from \eqref{supp2}, such that $\delta_j^{(n)} = \delta_1^{(n)}$ for all $j\le \kappa$. 
		If $\alpha:= qP_n(1/\theta_j)\notin \Z$, we obtain that $e^{2\pi i \alpha}$ is an eigenvalue, exactly as in Case 1b. If on the other hand, we have $\alpha = qP_n(1/\theta_j) \in \Z$, for $j\le \kappa$, then we obtain that
		$\alpha = qP_n(1/\theta_j)$ for all $j\le d$ (applying the corresponding automorphism $\tau_j$). It follows 
		from \eqref{supp2} that $\delta_1^{(n)} = 0$, and then $\delta_j^{(n)}=0$ for all $j\le d$ and hence $\beps_n=0$. But then $A^{n+k} \om{\bf 1} \in L\subset \Z^d$ for all $k\ge 1$, and we 
		obtain that $e^{2\pi i \om}$ is an eigenvalue for the system $(X_\zeta,T,\mu)$.

		\medskip

		Combining Lemma~\ref{lem-easy1} with our estimates \eqref{ineq3} and \eqref{ineq30}, we immediately obtain the following.
		
		\begin{prop} \label{prop:main} Suppose that $\zeta$ is a primitive aperiodic substitution with an irreducible over $\Q$ substitution matrix $\Sf$ and no eigenvalues of $\Sf$ on the unit 
			circle. Suppose furthermore that the substitution $\Z$-action is weakly mixing. 
			Let $\theta$ be the Perron-Frobenius eigenvalue of $\Sf$, and let $\bbe_1^*$ be the Perron-Frobenius eigenvector for $\Sf^{\sf T}$, such that its entries belong to $\Z[\theta]$.
			Let $q\in \N$ be such that $\langle {\bf 1}, \bbe_1^*\rangle^{-1} \in q^{-1}\Z[\theta]$.
			Then there exists $r_0>0$ and $C_5>0$, depending only on the substitution $\zeta$, such that 
			for any cylindrical function $f$ the following holds: 
			\begin{enumerate}
				\item[{\bf (i)}] If $\om\not\in q^{-1}\Z$, then
				\be \label{ineq41}
				\sig_f(B_r(\om)) \le C_5 \|f\|^2_\infty\cdot (\log(1/r))^{-2\gam} \cdot \log\left(\frac{1}{\dist(\om,q^{-1}\Z)}\right)^{2\gam},\ \ \mbox{for all}\ r\in (0,r_0).
				\ee

				\item[{\bf (ii)}] If $\om \in q^{-1}\Z\cap (0,1)$, then
				\be \label{ineq42}
				\sig_f(B_r(\om)) \le C_5 \|f\|^2_\infty\cdot (\log(1/r))^{-2\gam},\ \ \mbox{for all}\ r\in (0,r_0).
				\ee
			\end{enumerate}
		\end{prop}
		
		For $\om=0$ we need to assume that $\int f\,d\mu=0$. Then we obtain by \cite{Adam} and Lemma~\ref{lem-easy1} that
		\be \label{ineq43}
		\sig_f(B_r(0)) \le \const\cdot \|f\|_\infty^2 \cdot r^{\log |\theta_2|/\log\theta} \le C_5 \|f\|_\infty^2 \cdot (\log(1/r))^{-2\gam},\ \ \mbox{for all}\ r\in (0,r_0),
		\ee 
		where $\theta_2$ is the second largest in modulus eigenvalue of $\Sf$ and $C_5$ is adjusted appropriately. (Recall that $|\theta_2|>1$ by assumption.)
		
		\medskip
		
		Note that  the estimate \eqref{ineq41} ``blows up'' at the points in $q^{-1}\Z$, but ``gluing'' it with \eqref{ineq42}, \eqref{ineq43} will achieve the
		uniform bound \eqref{goal1}. 
		Combining \eqref{ineq42} and  \eqref{ineq43} yields the same inequality for all  $\om\in q^{-1}\Z$
		(note that $\om=1$ is identified with $\om=0$ on the circle). We can assume that 
		$r_0 < \frac{1}{2q}$. Let $\om_0\in q^{-1}\Z$ be such that $|\om-\om_0| = \dist(\om,q^{-1}\Z)$. If
		$$
		\log\frac{1}{|\om-\om_0|} \le \Bigl(\log\frac{1}{r}\Bigr)^{\half},
		$$
		which means $r\ll |\om-\om_0|$,
		then \eqref{ineq41} yields
		$$
		\sig_f(B_r(\om)) \le C_5 \|f\|^2_\infty\cdot (\log(1/r))^{-\gam}.
		$$
		If $\log\frac{1}{|\om-\om_0|} > \sqrt{\log\frac{1}{r}}$ and $r\le |\om-\om_0|$, then $B(\om,r) \subset B\bigl(\om_0, 2|\om-\om_0|\bigr)$, and we obtain from \eqref{ineq42} or  \eqref{ineq43}:
		\begin{eqnarray*}
			\sig_f(B_r(\om)) \le C_5 \|f\|_\infty^2 \cdot \Bigl(\log\frac{1}{2|\om-\om_0|}\Bigr)^{-2\gam} & \le & \wt C_5 \|f\|_\infty^2 \cdot \Bigl(\log\frac{1}{|\om-\om_0|}\Bigr)^{-2\gam} \\[1.1ex]
			& \le & \wt C_5 \|f\|_\infty^2 \cdot \bigl(\log(1/r)\bigr)^{-\gam}.
		\end{eqnarray*}
		If $r>|\om-\om_0|$, then $B(\om,r) \subset B(\om_0,2r)$, and we obtain from \eqref{ineq42} or  \eqref{ineq43}:
		$$
		\sig_f(B_r(\om)) \le C_5 \|f\|_\infty^2 \cdot (\log(1/2r))^{-2\gam} \le \wt C_5 \|f\|_\infty^2 \cdot  (\log(1/r))^{-2\gam}.
		$$
		Combining the last three inequalities implies \eqref{goal1}, completing the proof.
	\end{proof}
	
	
	
	\section{The Salem case} \label{sec:Salem}
	\subsection{Dynamics of Salem automorphisms}
	In this section we will focus on the case when the substitution matrix has eigenvalues on the unit circle; more specifically, that it is a Salem matrix.
	Just as in Section \ref{sec:approx}, we will first prove a general result on the dynamics of a vector of the form $\omega\mathbf{1}$,
	where now we assume $\omega\in (0,1)\cap \Q(\alpha)$, under the action of a toral automorphism defined by an irreducible matrix such that its characteristic polynomial is the 
	minimal polynomial of a Salem number.
	We set the notation in what follows.
	
	\medskip
	
	Let $A\in GL(d,\Z)$ be an irreducible (over $\Q$) matrix whose eigenvalues are  $\alpha_1=\alpha > 1$, a Salem number of degree $d=2m+2$, a single real conjugate $\alpha_2 = \alpha^{-1}$, and complex conjugates of $\alpha$ given by $\alpha_3 = \overline{\alpha_4} = e^{2\pi i \theta_1}, \dots, \alpha_{d-1} = \overline{\alpha_d} = e^{2\pi i \theta_m}$. Denote by $\sigma_j:\Q(\alpha) \hookrightarrow \C$ the Galois embedding sending $\alpha$ to $e^{2\pi i \theta_j}$, for $j=1,\dots,m$.
	
	\medskip
	Let $\omega\in (0,1)\cap \Q(\alpha)$. Decompose the vector $\mathbf{1}$ with respect to an eigenbasis $\{\bbe_k\}^d_{k=1}$ of $A$ (which is diagonalizable over $\C$ by the irreducibility of the characteristic polynomial), where each $\bbe_k \in \Q(\alpha_k)^d$, $k=1,\dots,d$, is the eigenvector associated to $\alpha_k$. Furthermore, we may take $\bbe_k = \overline{\bbe_{k+1}}$ for $k = 3,5,\dots,d-1$, and  $\bbe_{2j+1} = \sigma_j(\bbe_1)$, $j=1,\ldots,m$, where the embedding is applied coordinate-wise. We get
	\[
	\omega A^{n} \mathbf{1} = \omega C_1\alpha^n\bbe_1 + \omega C_2\alpha^{-n}\bbe_2 + \omega\sum_{k=3}^d C_k\alpha_k^n \bbe_k.
	\]
	Consider $\{\bbe_k^*\}^d_{k=1}$, a dual basis of $\{\bbe_k\}^d_{k=1}$ consisting of eigenvectors of $A^{\sf T}$, that is, $A^{\sf T}\bbe_k^* = \overline{\alpha_{k}}\bbe_k^*$ and
	\[
	\langle \bbe_j,\bbe_k^*\rangle = \delta_{j,k}.
	\]
	Note that $\bbe_k^* \in \Q(\overline{\alpha_k})^d= \Q(\alpha_k)^d$ (the non-real conjugates are of modulus one, hence $\overline{\alpha_k} = \alpha_k^{-1}$) and that we can take $\bbe_k^* = \overline{\bbe_{k+1}^*}$ for $k = 3,5,\dots,d-1$. Now we can express the coefficients as
	\[
	C_k = \dfrac{\langle \mathbf{1},\bbe_k^* \rangle}{\langle \bbe_k,\bbe_k^* \rangle} = \langle \mathbf{1},\bbe_k^* \rangle.
	\]
	Observe that $C_k = \overline{C_{k+1}}$ for $k = 3,5,\dots,d-1$. The coefficient $C_1 \in \Q(\alpha)$ does not vanish, since $A$ is irreducible. This implies that each $C_k$ is not zero since each $\bbe_k^* = \sigma(\bbe_1^*)$ for a suitable embedding $\sigma$.
	In fact, $C_{2j+1} = \sig_j(C_1)$ for $j=1,\ldots,m$, and $C_2 = \sig_0(C_1)$, where $\sig_0$ is the  Galois embedding of $\Q(\alpha)$ sending $\alpha$ to $\alpha^{-1}$.
	
	\medskip
	Let $\Gamma < \Z^d$ be an integer lattice such that $A\Gamma \subset \Gamma$, which is generated by $\mathbf{v}_1,\dots,\mathbf{v}_r$. We will be interested in the distribution of $\langle \omega A^n \mathbf{1},\mathbf{v}_i \rangle$ modulo $\Z$. By irreducibility of $A$, $\langle \bbe_1,\mathbf{v}_i \rangle$ is not zero for any $i=1,\dots,r$. Let $\mathbf{v}$ be any of the $\mathbf{v}_i$, then
	\[
	\langle \omega A^n \mathbf{1},\mathbf{v} \rangle = \underbrace{\omega C_1\langle \bbe_1,\mathbf{v} \rangle}_{=:\eta^{\mathbf{v}} \in \Q(\alpha) }\alpha^n + \omega C_2 \langle \bbe_2,\mathbf{v} \rangle \alpha^{-n} + \omega \sum_{k=3}^d C_k \langle \bbe_k,\mathbf{v} \rangle \alpha_k^n,
	\]
	and $\eta^{\mathbf{v}} \neq 0$. We can write $\langle \bbe_k,\mathbf{v} \rangle = Q^{\mathbf{v}}(\alpha_k)$ for some non-trivial polynomial $Q^{\mathbf{v}} \in \Q[X]$ of degree at most $d-1$: since $\langle \bbe_1,\mathbf{v} \rangle$ is not zero, for any embedding $\sigma$, as before, we have $\langle \bbe_k,\mathbf{v} \rangle = \sigma(\langle \bbe_1,\mathbf{v} \rangle) \neq 0$. 
	Then the last formula can be written as
	\be \label{new1}
	\langle \omega A^n \mathbf{1},\mathbf{v} \rangle = \eta^\bv\alpha^n + \om\sig_0(C_1)Q^{\mathbf{v}}(\alpha_2) \alpha^{-n} + 
	\om\sum_{j=1}^m 2\re\left[\sig_j(C_1) Q^{\mathbf{v}}(\alpha_{2j+1})\alpha_{2j+1}^n\right],
	\ee
	where $$\eta^{\mathbf{v}} = \om C_1 Q^{\mathbf{v}}(\alpha)  = \om\langle \mathbf{1},\bbe_1^*\rangle \langle \bbe_1,\mathbf{v} \rangle \in \Q(\alpha).$$
	Let 
	\[
	\eta^{\mathbf{v}} = \dfrac{l_0 + l_1\alpha + \dots + l_{d-1}\alpha^{d-1}}{L},
	\]
	where $l_j \in \Z, L\geq1 \text{ and } \gcd(l_0,\dots,l_{d-1},L) = 1$. Note that $L$ does not depend on the vector $\mathbf{v}$, since the latter is an integer vector. Recall the definition of \textit{trace} of an algebraic number: for $\eta\in\Q(\alpha)$, set $\Tr(\eta) = \sum_{\sigma:\Q(\alpha) \hookrightarrow \C} \sigma(\eta)$, where the sum runs over all $d$ embeddings. In particular, $\Tr(\eta)\in\Z$ for $\eta\in\Z[\alpha]$.
	Then,
	\be \label{new2}
	\Z\ni\Tr(L\eta^{\mathbf{v}}\alpha^n) = L\eta^{\mathbf{v}}\alpha^n + L\sig_0(\eta^{\mathbf{v}}) \alpha^{-n} + L \sum_{j=1}^m 2\re\left[\sig_j(\eta^{\mathbf{v}})\alpha_{2j+1}^n\right].
	\ee
	Observe that
	$$
	\sig_0(\eta^{\mathbf{v}}) = \sig_0(\om) \sig_0(C_1) Q^{\mathbf{v}}(\alpha^{-1}),\ \ \ \sig_j(\eta^{\mathbf{v}}) = \sig_j(\om) \sig_j(C_1) Q^{\mathbf{v}}(\alpha_{2j+1}),\ j=1,\ldots,m.
	$$
	It follows from \eqref{new1} and \eqref{new2} that 
	\begin{eqnarray} \nonumber
		L \langle \omega A^n \mathbf{1},\mathbf{v} \rangle & =& \Tr(L\eta^{\mathbf{v}}\alpha^n) + L(\om-\sig_0(\om))Q^{\mathbf{v}}(\alpha_2)\alpha^{-n} \\ \nonumber &  & + 
		2L\sum_{j=1}^m \re\left[(\om-\sig_j(\om))\sig_j(C_1) Q^{\mathbf{v}}(\alpha_{2j+1})\alpha_{2j+1}^n\right] \\
		& = & \Tr(L\eta^{\mathbf{v}}\alpha^n) + L(\om-\sig_0(\om))Q^{\mathbf{v}}(\alpha_2)\alpha^{-n} + L\sum_{j=1}^m  \mathfrak{H}_j \cos(2\pi \theta_j + \phi_j), \label{new3}
	\end{eqnarray}
	where
	\be \label{frakh}
	\mathfrak{H}_j = 2\bigl|{(\om-\sig_j(\om))\sig_j(C_1) Q^{\mathbf{v}}(\alpha_{2j+1})}\bigr|\ \ \ \mbox{and} \ \ \ \phi_j = \arg\bigl({(\om-\sig_j(\om))\sig_j(C_1) Q^{\mathbf{v}}(\alpha_{2j+1})}\bigr).
	\ee

	Observe that $\bigl(\Tr(L\eta^{\mathbf{v}}\alpha^n)\bigr)_{n\ge 0}$ is an integer sequence 
	satisfying a linear recurrence with a characteristic polynomial equal to the minimal polynomial of $\alpha$ (and the characteristic polynomial of $A$). It is shown in in the course of the proof of \cite[Proposition 4.1]{Marshall20},
	similarly to \cite{Dub}, that $\bigl(\Tr(L\eta^{\mathbf{v}}\alpha^n) \text{ (mod }L)\bigr)_{n\ge 0}$ is a strictly periodic sequence of some period  $1\le P(\mathbf{v}) \le L^d$.
	Let $$R_n = \sum^m_{j=1} \mathfrak{H}_j\cos(2\pi n\theta_j + \phi_j).$$
	We obtain from \eqref{new3} 
	that there exists $1\le P(\mathbf{v}) \le L^d$, a positive integer, such that for every $0\leq p \leq  P(\mathbf{v})-1$ there is an integer $0\leq a^{\mathbf{v}}_p\leq L-1$, such that for all $n\geq 0$,
	\be \label{new4}
	\langle \omega A^n \mathbf{1},\mathbf{v} \rangle \text{ (mod 1) } = 
	\dfrac{a^{\mathbf{v}}_p}{L} + (\om-\sig_0(\om))Q^{\mathbf{v}}(\alpha_2)\alpha^{-n} + R_n \text{ (mod 1)},\ \ \ \mbox{for $n\equiv p \text{ (mod } P(\mathbf{v}))$}.
	\ee
	
	We are interested in the distribution of the sequence in the last equation. Given an interval $J\subset [0,1]$, we have 
	\begin{align*} 
		&\lim_{N\to\infty} \dfrac{1}{N}\#\bigl\{n\leq N \,:\, \langle \omega A^n \mathbf{1},\mathbf{v} \rangle \text{ (mod 1) } \in J \bigr\} \\
		&= \lim_{N\to\infty} \dfrac{1}{ P(\mathbf{v})}\sum_{p=0}^{ P(\mathbf{v})-1}\dfrac{1}{N/P(\mathbf{v})} \#\bigl\{n\leq N, n\equiv p \text{ (mod } P(\mathbf{v})) \,:\, \langle \omega A^n \mathbf{1},\mathbf{v} \rangle \text{ (mod 1) } \in J \bigr\} \\
		&= \lim_{N\to\infty}   \dfrac{1}{ P(\mathbf{v})}\sum_{p=0}^{ P(\mathbf{v})-1} \dfrac{1}{N/ P(\mathbf{v})}\#\left\{n\leq N, n\equiv p \text{ (mod } P(\mathbf{v})) \,:\, \dfrac{a^{\mathbf{v}}_p}{L} + R_n \text{ (mod 1) } \in J \right\}.
	\end{align*} 

	Denoting by $\mathds{1}_J$ the $\Z$-periodic extension of the indicator of an interval $J\subset[0,1]$ we conclude that 
	\begin{equation} \label{limit}
		\lim_{N\to\infty} \dfrac{1}{N}\#\bigl\{n\leq N \,:\, \langle \omega A^n \mathbf{1},\mathbf{v} \rangle \text{ (mod 1) } \in J \bigr\} = 
		\dfrac{1}{P(\mathbf{v})} \sum_{p=0}^{P(\mathbf{v})-1} \int_{(\R/\Z)^m} \mathds{1}_{_J}\left(\dfrac{a^{\mathbf{v}}_p}{L} + \mathfrak{R}(\vec{x})\right)d\vec{x}, 
	\end{equation} where
	\[
	\mathfrak{R}(\vec{x}) = \sum_{j=1}^m \mathfrak{H}_j \cos(2\pi x_j),
	\]
	since $\bigl((n\theta_1,\ldots, n\theta_m)\bigr)_{n\ge 1}$ is uniformly distributed modulo $\Z^m$, see \cite[Lemma 3.8]{Bug12}, and the sequences involved have continuous distribution functions. This integral is similar to the one treated in \cite{Marshall20}.
	however, unlike in
	\cite{Marshall20}, it is not obvious that $\mathfrak{R}(\vec{x})\not\equiv 0$.
	Recall that our construction depends on the vector $\mathbf{v}$, so we will write $R_n(\mathbf{v}), \mathfrak{R}_\mathbf{v}(\vec x), \mathfrak{H}_j(\mathbf{v})$ below. Let
	$$
	\mathfrak{H}(\mathbf{v}) = \max_{\vec x\in \R^m} |\mathfrak{R}_\mathbf{v}(\vec{x})| = \sum_{j=1}^m \mathfrak{H}_j(\mathbf{v}).
	$$

	\begin{lemma} \label{lem:nonde}
		If $\omega\notin \Q$, then $\mathfrak{H}(\mathbf{v}_i)\ne 0$ for any  $\mathbf{v}_i$ from the set of free generators of the lattice $\Gam<\R^d$.
	\end{lemma}
	
	\begin{proof} We argue by contradiction. If $\mathfrak{H}(\mathbf{v}_i)=0$ for some $i$, then $R_n(\mathbf{v}_i) = 0$ for all $n$.
		Hence, in view of \eqref{new4}, 
		$$
		\langle \omega A^n \mathbf{1},\mathbf{v}_i \rangle \text{ (mod 1) } = 
		\dfrac{a^{\mathbf{v}_i}_p}{L} + (\om-\sig_0(\om))Q^{\mathbf{v}_i}(\alpha_2)\alpha^{-n}  \text{ (mod 1)} \longrightarrow \dfrac{a^{\mathbf{v}_i}_p}{L},
		$$
		as $n\equiv p \text{ (mod }P)$ goes to infinity. Therefore,
		\be \label{eq:return}
		\langle L \omega A^n \mathbf{1},\mathbf{v}_i \rangle \text{ (mod 1) } \longrightarrow 0,\ \ n\to \infty.
		\ee
		By \cite[Lemma 2]{Kwapisz}, equation \eqref{eq:return} implies
		\[
		L\omega\mathbf{1} = \mathbf{q} + \mathbf{s}, \quad \mathbf{q}\in\Q^d,\; \mathbf{s}\in \text{Span}_\R(\bbe_2).
		\]
		But this means $\mathbf{s}=\mathbf{0}$, since otherwise there exists $\lambda\in\R$ different from zero such that $\lambda \bbe_2\in \Z^d$, which is absurd, since $\bbe_2$ has rationally independent entries by the irreducibility of the matrix. 
	\end{proof}
	
	\medskip
	
	The main technical ingredient of the proof of Theorem~\ref{th:Salem} is the next estimate. Recall that the periodic extension of the indicator function
	of an interval $J\subset[0,1]$ is denoted by $\mathds{1}_J$. For $0<\delta<1/2$, we will denote by $J(\delta)$ the interval $[\delta, 1-\delta]$.

	\begin{theorem}[\cite{Marshall20}]\label{Integral} 
		Let $G_1,\dots,G_m$ be non-negative real numbers, such that $\sum_{i=1}^m G_i \ge \Delta$ for some constant $\Delta>0$. Then there exists an  effective $\delta = \delta(\Delta) > 0$ such that for any $s_0 \ge 0$, 
		\[
		\int_{(\R/\Z)^m} \mathds{1}_{J(\delta)} \left( s_0 + \sum_{j=1}^m G_j \cos(2\pi x_j) \right) dx_1\dots dx_m \geq \dfrac{1}{2}. 
		\]
	\end{theorem}
Here, and in several places below, 	an ``effective'' constant means that it depends only on $|\om|, |\sig_0(\om)|$, and $L$, and
this dependence can be made explicit. It will also depend on the matrix $A$, of course.
	In \cite{Marshall20} the theorem is proved in Lemmas 4.3 and 4.4-4.7, with the constant $s_0$ equal to 
	$\ell/L$ for some integer $\ell$ (the only case needed here as well), but this does not change the 
	proof. Now, at least for $\om\notin \Q$, we can immediately apply Theorem~\ref{Integral} to the functions $\frac{a^{\bv_i}_p}{L} + \mathfrak{R}_{\mathbf{v}_i}(\vec x)$, for $p\in \{0,\ldots,P(\mathbf{v}_i)-1\}$, appearing in \eqref{limit}, with 
	$$\mathfrak{H}(\mathbf{v}_i)=\sum^m_{j=1} \bigl|{(\om-\sig_j(\om))\sig_j(C_1) Q^{\mathbf{v}_i}(\alpha_{2j+1})}\bigr|>0, \text{ if } \omega \notin \Q,$$
	in view of \eqref{frakh} and Lemma \ref{lem:nonde}. However, we do not have a good control on the size of $|\om-\sig_j(\om)|$. We would like to make our estimates effective, similarly to \cite{Marshall20}.
	To this end, we separate the analysis into cases, according to size of the
	constant $\max_{i=1,\dots,r}\mathfrak{H}(\mathbf{v}_i) =: \mathfrak{H}(\Gamma)$.\\
	
	Now we formulate and prove the main result of this section, an extension of \cite[Theorem 1.4]{Marshall20}.
	\begin{theorem}\label{thm:SalemDist}
		Let $A\in GL(d,\Z)$ be an irreducible matrix such that its characteristic polynomial is the minimal polynomial of a Salem number $\alpha$. Let $\omega\in(0,1)\cap\Q(\alpha)$. Then one of the next possibilities hold:
		\begin{itemize}
			\item[(i)] $\om\mathbf{1} \in \Gamma^*$.
			\item[(ii)] There exists an effective $\delta = \delta(\omega) > 0$ such that
			\[
			\lim_{N\to\infty} \dfrac{1}{N}\#\bigl\{n\leq N \,:\, \norm{\omega A^n \mathbf{1} }_{\R^d/\Gamma^*} \geq \delta \bigr\} \geq \dfrac{1}{2L^d},
			\]
			where $L$ is the denominator of $\om$ in minimal form. 
			Moreover, $\delta$ depends only on $\abs{\omega}$, $\abs{\sigma_0(\omega)}$, and $L$.
		\end{itemize}
	\end{theorem}

\begin{proof}

	{\bf Case 1:}  $\omega \in\Q$, i.e, $\mathfrak{H}(\Gamma)=0$. Put $\omega = l/L$, for $1\leq l \leq L$. Since the orbit under the toral  automorphism defined by $A$ of $\omega\mathbf{1}$ is periodic, for any integer vector $\mathbf{v}$, the sequence $(\langle \omega A^n \mathbf{1}, \mathbf{v} \rangle \text{ (mod 1) })_{n\geq 0}$ is purely periodic of period $P(\mathbf{v})\le L^d$. If this sequence is the null sequence for every $\mathbf{v}_i$, then by definition $\om\mathbf{1}\in\Gamma^*$,  i.e., the first possibility (i) of the theorem holds. Otherwise we conclude that for some $\mathbf{v}_i$ and $p \in \{0,\dots,P-1\}$ ($P = P(\mathbf{v}_i)$),
	\[
	\norm{\langle \omega A^{Pn + p} \mathbf{1}, \mathbf{v}_i \rangle}_{\R/\Z} \geq 1/L, \text{ for all } n\geq 0. 
	\]
	It follows that
	\begin{equation}\label{rational}
		\lim_{N\to\infty} \dfrac{1}{N}\#\bigl\{n\leq N \,:\, \langle \omega A^n \mathbf{1},\mathbf{v}_i \rangle \text{ (mod 1) } \in J(1/L) \bigr\} \geq \dfrac{1}{L^d},
	\end{equation}
	which implies (ii), by Lemma~\ref{lem:lattice}.
	
	\medskip
	
	From now on we suppose $\omega \notin \Q$. In view of Lemma~\ref{lem:lattice}, 
		\be \label{eq:latt2}
	\max_{i=1,\dots,r} \, \norm{\langle \omega A^n \mathbf{1},\mathbf{v}_i \rangle}_{\R/\Z} \geq C_\Gam^{-1}\norm{\omega A^n \mathbf{1}}_{\R/\Gamma^*},
	\ee
	for some explicit constant $C_\Gam>1$ depending only on the lattice (and the choice of generators). 
	 Recall that the period $P = P(\mathbf{v}_i)$ depends on the integer vector $\mathbf{v}_i$. Since in what follows  we fix $\mathbf{v}_i$, we will omit the dependence on $\bv_i$ from the notation. Let $B=A^P$, and denote its eigenvalues by $\beta_j = \alpha^P_j$.
	Fix $p\in \{0,\ldots,P-1\}$ and write, similarly to Section~\ref{sec:approx}:
	\be \label{construct}
	\omega B^n A^p\mathbf{1} = \mathbf{p}_n + \beps_n,
	\ee
	with $\mathbf{p}_n$ being the nearest element in $\Gamma^*$. It follows from Lemma~\ref{lem1}(ii) and \eqref{eq:latt2} that for an appropriate explicit constant $c=c_\Gam$ holds the implication
	$$
	\max_{i=1,\dots,r} \norm{\langle \omega A^n \mathbf{1},\mathbf{v}_i \rangle}_{\R/\Z} < c \implies \beps_{n+1} = B \beps_n \quad \
	\mbox{for all} \ \mathbf{v}_i.
	$$

	{\bf Case 2:}   $0<\mathfrak{H}(\Gamma) < c$ and $a^{\mathbf{v}_i}_p=0$ for all $i=1,\dots,r$. Let $\mathbf{v}$ be any of the generators $\mathbf{v}_i$. Then \eqref{new4} implies, in view of $\mathfrak{H}(\mathbf{v}) \ge |R_n(\mathbf{v})|$, that
	\begin{equation}\label{Rec}
		\beps_{n_1+n} = B^n\beps_{n_1},
	\end{equation}
	for all $n\geq0$, for some $n_1$ (depending on $\abs{\omega}$ and $\abs{\sigma_0(\omega)}$). Now, equation \eqref{eq3} becomes, for $k\le d$:
	\[
	\omega\langle \mathbf{1},\bbe_k^* \rangle = \underbrace{ \langle\mathbf{p}_{0},\bbe_k^* \rangle - \dfrac{b_1^{(k)}}{\beta_k} -\dots-\dfrac{b_{n_1}^k}{\beta_k^{n_1}} }_{=: \langle \mathbf{1},\bbe_k^* \rangle P_{n_1}(1/\beta_k)} + \underbrace{ \dfrac{\langle \beps_{n_1},\bbe_k^* \rangle}{\beta_k^{n_1}} }_{=: \langle \mathbf{1},\bbe_k^* \rangle \delta_k^{(n_1)}}.
	\]
	Repeating the steps after equation (\ref{eq3}), we get an analog of \eqref{eq31}:
	\begin{equation} \label{Polynom}
		P_{n_1}(1/\beta_k) - P_{n_1}(1/\beta) = \delta_1^{(n_1)} - \delta_k^{(n_1)} = - \delta_k^{(n_1)}, 
	\end{equation}
	since we claim that  $\delta_1^{(n_1)} = 0$. Indeed, for any $n\geq0$ we have by \eqref{Rec}:
	\begin{align*}
		\beta^n\abs{\langle\beps_{n_1},\bbe_1^{\:*}\rangle} &= \abs{\langle B^n\beps_{n_1},\bbe_1^{\:*}\rangle}\\
		& = \abs{\langle\beps_{n_1+n},\bbe_1^{\:*}\rangle}\\
		&\leq {\norm{\beps_{n_1+n}}}_2\cdot{\norm{\bbe_1^{\:*}}}_2,
	\end{align*}
	which leads to a contradiction if $\delta_1^{(n_1)} \neq 0$, as $n$ goes to infinity.\\
	
	Decomposing $\beps_{n_1}$ with respect to the eigenbasis of $B$ (the same one as for $A$), we obtain
	\[
	\beps_{n_1+n} = D_1(\omega,n)\beta^n \bbe_1 + D_2(\omega,n_1)\beta^{-n}\bbe_2 + \sum_{j=1}^m  2\re\left( D_j(\omega,n_1)e^{2\pi i nP\theta_j}\bbe_{2j+1}\right).
	\]
	Note that $D_1(\omega,n_1) = \langle\beps_{n_1}, \bbe_1^{*} \rangle  = 0$ as we showed above. We obtain 
	that for all $n\geq0$,
	\begin{eqnarray*}
		& & \langle \beps_{n_1+n}, \mathbf{v} \rangle \text{ (mod 1) } \\
		& = & D_2(\omega,n_1) \langle \bbe_2, \mathbf{v}\rangle\beta^{-n} + \sum_{j=1}^m 2\re\left( D_j(\omega,n_1)\langle \bbe_{2j+1},\mathbf{v}\rangle \beta_j^{n} \right) \text{ (mod 1) }\\
		& = & D_2(\omega,n_1)\langle \bbe_2, \mathbf{v}\rangle\beta^{-n} + 2\sum_{j=1}^m \abs{D_j(\omega,n_1)} \cdot \abs{\langle \bbe_{2j+1},\mathbf{v}\rangle} \cos(2\pi nP\theta_j + \xi_j) \text{ (mod 1) },
	\end{eqnarray*}
	where $\xi_j\in\R$. By \eqref{construct}, for all $n\geq 0$,
	\[
	\langle \omega B^{n_1+n} A^p\mathbf{1},\mathbf{v} \rangle \text{ (mod 1) } = \langle \beps_{n_1+n}, \mathbf{v} \rangle \text{ (mod 1) }.
	\]
	Comparing this equality with \eqref{new4} we can deduce that
	\[
	\mathfrak{H}_j(\mathbf{v})  = \abs{D_j(\omega,n_1)} \cdot \abs{\langle \bbe_{2j+1},\mathbf{v}\rangle} = \abs{\langle \mathbf{1},\bbe_k^* \rangle\delta_{2j+1}^{(n_1)}} \abs{\langle \bbe_{2j+1},\mathbf{v}\rangle},
	\]
	since the amplitudes of asymptotically equal oscillating sums must match.\\

	Since $\mathfrak{H}(\mathbf{v}) \neq 0$, the last equation shows that some  $\delta_{2j+1}^{(n_1)}$ is not equal to zero by Lemma \ref{lem:nonde}. In view of \eqref{Polynom}, we have an explicit lower bound in terms of $\beta$, $|\om|$ and $|\sig_0(\om)|$ for $\abs{\delta_{2j+1}^{(n_1)}}$, which gives us a lower bound for $\mathfrak{H}_j(\mathbf{v})$: by Lemma \ref{hochman}, there exists a constant $s>0$ only depending on $\beta$ such that 
	\[
	\abs{P_{n_1}(1/\beta_k) - P_{n_1}(1/\beta)} \geq s^{n_1}.
	\]
	In other words, for some constant $t = t(A,\mathbf{v})>0$,
	\[
	\mathfrak{H}_j(\mathbf{v}) \geq ts^{n_1}.
	\]
	Since $\beta = \alpha^P$ and $1\leq P \leq L^d$, the dependence on $\beta$ may be changed to a dependence on $\alpha$ and $L$. To summarize, 
	we may use Theorem \ref{Integral} with $\Delta = \Delta(\abs{\om},\abs{\sigma_0(\om)},L) = ts^{n_1}$.\\

	{\bf Case 3: } $0<\mathfrak{H}(\Gamma) < c$ and $a^{\mathbf{v}_i}_p\neq0$, for some $\mathbf{v}_i$. We may suppose $\mathfrak{H}(\Gamma)$ is smaller than $1/2L$, otherwise we conclude using Theorem \ref{Integral} with $\Delta = 1/2L$. Recall that $\mathfrak{H}(\mathbf{v}_i) = \max_{\vec{x}\in[0,1]^d} \abs{\mathfrak{R}_{\mathbf{v}_i}(\vec{x})}$. In this case,
	\[
	\dfrac{1}{4L} \leq \dfrac{-1/2+1}{L}\leq \dfrac{-\mathfrak{R}_{\mathbf{v}_i}(\vec{x})+a^{\mathbf{v}_i}_p}{L} \leq \dfrac{1/2+L-1}{L} \leq 1-\dfrac{1}{4L}.
	\]
	That is, the integral of equation \eqref{limit} is equal to one for the interval $J(\delta)$ with $\delta = 1/4L$.\\
	
	{\bf Case 4: }  $\mathfrak{H}(\Gamma) \geq c$. In this case we can apply Theorem \ref{Integral} with $\Delta = c$, which is independent of $\om$.

	Cases 2-4 above lead us to the same conclusion: we obtain the second claim in the theorem simply by equation \eqref{limit}.
\end{proof}

	\subsection{Application to Salem substitutions}
	Let $\zeta$ be an aperiodic primitive substitution such that its substitution matrix is irreducible and its Perron-Frobenius eigenvalue is a Salem number $\alpha$. As in Section~\ref{ProofSec}, we can assume without loss of generality that all return words for $\zeta$ of less than certain length are good return words. Let $\Gam$ be the lattice generated by
	population vectors of return words. Then $\Gam=\Z[\ell(v_1),\ldots,\ell(v_k)]$ for some return words $v_1,\ldots,v_k$, with $k\ge d$.\\
	
	Now it is easy to conclude the proof of Theorem~\ref{th:Salem}, similarly to Section \ref{ProofSec}. 
	Let $\omega\in(0,1)\cap\Q(\alpha)$ and $A=\Sf_\zeta^{\sf T}$. If $\omega\in \Q$, then we cannot have $(\langle \omega A^n \mathbf{1}, \ell(v_i) \rangle \text{ (mod 1) })_{n\geq 0}$ identically equal to zero for every $v_i$, that is, $\om\mathbf{1}\in\Gamma^*$, since in that case the same its true for an arbitrary return word, which would imply that 
	$\omega$ is an eigenvalue of $(X_\zeta,T,\mu)$ by Theorem \ref{th-host}. This is impossible since Salem substitution $\Z$-actions are weakly-mixing.\\
	
	By Theorem \ref{thm:SalemDist}, there exists $\delta = \delta(\abs{\omega},\abs{\sigma_0(\omega)},L)>0$ such that for all $N$ sufficiently large holds
	\begin{equation}
	\dfrac{1}{N}\#\bigl\{n\leq N \,:\, \norm{\omega A^n \mathbf{1} }_{\R^d/\Gamma^*} \geq \delta \bigr\} \geq \dfrac{1}{3L^d}.
	\end{equation}

		\\

	Then, by the product bound on the twisted Birkhoff sum in Proposition~\ref{prop:BuSo}, for any cylindrical function $f$ and any $x\in X_\zeta$,
	\begin{align*}
		\abs{S^x_N(f,\omega)} &\leq O(1)\cdot {\|f\|}_\infty \cdot N (1-c_3\delta^2)^{\log_{\alpha}(N)/3L^d} \\
		&= O(1)\cdot {\|f\|}_\infty \cdot N(\alpha^{\log_\alpha(1-c_3\delta^2)})^{\log_\alpha(N)/3L^d} \\
		&= O(1)\cdot {\|f\|}_\infty \cdot N^{1+ \log_{\alpha}(1-c_3\delta^2)/3L^d} \\
		&= O(1)\cdot {\|f\|}_\infty \cdot N^{\wt{\vartheta}}, \quad \wt{\vartheta} \in (0,1).
	\end{align*}
	From this we conclude   the proof of Theorem \ref{th:Salem} by a direct use of Lemma \ref{lem-easy1}.
	

	\section{Non-uniformity of the H\"older exponent and the proof of Theorem~\ref{th:lower}} \label{sec:lower}
	
	Both in Theorem \ref{th:Salem} and Theorem 1.1 of \cite{Marshall20} we find a H\"older exponent for the spectral measures at parameters which belong
	to the dense set $\Q(\alpha)$. This H\"older exponent depends on $\om$. In this section we show that is not possible to get a uniform exponent for all $\omega\in \Q(\alpha)$ for the $\Z$-action.
	In fact, the same proof works for the self-similar $\R$-action, which  is a consequence of the next well-known property of Salem numbers.
	\begin{prop}[see \cite{Bug12}, Theorem 3.9]\label{prop:SalemProp}
		Let $\alpha$ be a Salem number and $\eps>0$. Then there exists $\eta \in \Q(\alpha)$ different from zero such that 
		\[
		\norm{\eta\alpha^n}_{\R/\Z} < \eps,
		\]
		for all $n\geq0$.
	\end{prop}
	For the $\Z$-action we need a multi-dimensional version of the last proposition.
	\begin{lemma}\label{SalemPropMatrix}
		Let $A\in GL(d,\Z)$ be an irreducible matrix such that its characteristic polynomial is a 
		minimal polynomial of a Salem number $\alpha$, and let $\eps > 0$. There exists $\eta\in\Z[\alpha]\cap(0,1)$ such that
		\[
		\norm{\eta A^n {\bf 1}}_{\R^d/\Z^d} < \eps,
		\]
		for all $n\geq0$.
	\end{lemma}
	
	We continue with the set-up of Section~\ref{sec:Salem}. Namely, 
	$\alpha_1=\alpha>1$ is a Salem number of degree $d=2m+2$, with a single real conjugate $\alpha_2 = \alpha^{-1}$, 
	and complex conjugates of $\alpha$ given by $\alpha_3 = \overline{\alpha_4} = e^{2\pi i \theta_1}, \dots, \alpha_{d-1} = \overline{\alpha_d} = e^{2\pi i \theta_m}$. 
	The symbol $\sigma_j:\Q(\alpha) \hookrightarrow \C$ denotes the Galois embedding sending $\alpha$ to $e^{2\pi i \theta_j}$, for $j=1,\dots,m$, and 
	$\sigma_0:\Q(\alpha) \hookrightarrow \R$ is the Galois embedding sending $\alpha$ to $\alpha^{-1}$.

	\begin{proof}[Proof of Lemma~\ref{SalemPropMatrix}]
		The proof is similar to that of Proposition~\ref{prop:SalemProp}.
		We claim that it suffices to show that for any  $\eps>0$ there exists $\eta\in \Z[\alpha]$ such that for all $j=0,\dots, m$,
		\begin{equation}\label{SuffCond}
			\abs{\eta - \sigma_j(\eta)} < \eps.
		\end{equation}
		Indeed, 
			$$
			\norm{\eta A^n {\bf 1}}_{\R^d/\Z^d}\asymp \max_{\bv \in \Z^d} \norm{\langle \eta A^n {\bf 1}, \bv \rangle}_{\R/\Z},
			$$
			and then equations \eqref{new3} and \eqref{frakh},  with $\om = \eta$ and $L=1$, in which $\bv$ can be replaced by an arbitrary integer vector,
			imply the claim.
			
		Consider the following embedding of $\Q(\alpha)$ into $\R^2\oplus \C^m \simeq \R^d$ (we identify $\C$ with $\R\oplus\R$):
		\begin{align*}
			\tau:\omega &\mapsto \begin{pmatrix}
				\om,&
				\om - \sigma_0(\om),&
				\om - \sigma_1(\om),&
				\om - \sigma_3(\om),&
				\dots,&
				\om - \sigma_{d-1}(\om)
			\end{pmatrix} \\
			&\simeq \begin{pmatrix}
				\om,&
				\om - \sigma_0(\om),&
				\re(\om - \sigma_1(\om)),&
				\im(\om - \sigma_1(\om)),&
				\dots,&
				\im(\om - \sigma_{d-1}(\om))
			\end{pmatrix}
		\end{align*}
		The image of $\Z[\alpha]$ under this map is a full-rank lattice in $\R^d$. Indeed, we just need to verify that for a basis of $\Z[\alpha]$, its image is a basis for $\R^d$. We may use the power basis given by $\{1,\alpha,\dots,\alpha^{d-1}\}$. It suffices to show that the determinant of the next square real matrix of dimension $d$ is different from zero:
		\[
		M = \begin{pmatrix}
			1 & 0 & 0 & 0 & \cdots & 0 \\
			\alpha & \alpha - \alpha^{-1} & \re(\alpha - \alpha_3) & \im(\alpha - \alpha_3) & \cdots & \im(\alpha - \alpha_{d-1}) \\
			\vdots & \vdots & \vdots & \vdots & \vdots & \vdots \\
			\alpha^{d-1} & \alpha^{d-1} - \alpha^{1-d} & \re(\alpha^{d-1} - \alpha_3^{d-1}) & \im(\alpha^{d-1} - \alpha_3^{d-1}) & \cdots & \im(\alpha^{d-1} - \alpha_{d-1}^{d-1})
		\end{pmatrix}
		\]
		(this is just the identification of the matrix $(\tau(1),\tau(\alpha),\dots,\tau(\alpha^{d-1}))^{\sf T}$). By elementary operations on the columns (which we can perform as vectors in $\C^d$, and it does not have an effect on whether the determinant vanishes or not), we obtain the complex matrix
		\[
		\widetilde{M} = \begin{pmatrix}
			1 & 1 & 1 & 1 & \cdots & 1 \\
			\alpha & \alpha^{-1} & \alpha_3 &  \alpha_4 & \cdots & \alpha_d \\
			\vdots & \vdots & \vdots & \vdots & \vdots & \vdots \\
			\alpha^{d-1} &  \alpha^{1-d} & \alpha_3^{d-1} &  \alpha_4^{d-1} & \cdots & \alpha_d^{d-1}
		\end{pmatrix},
		\]
		whose determinant is not equal to zero, since it is the Vandermonde matrix of different numbers.
		
		For any $\eps>0$, let $E>0$ which will be set later. Consider the system of inequalities in the variables $n_i$, where $\tau_j(\omega) = \omega - \sigma_j(\omega)$, $j=0,1,\dots,m$.
		\[
		\left|\sum_{i=0}^{d-1} n_i \alpha^i\right| < E, \quad \left| \sum_{i=0}^{d-1} n_i \tau_j(\alpha^i) \right| < \eps, \quad \text{ for } j=0,1,\dots,m.
		\]
		By ``Minkowski's first theorem" (see \cite{Bug04}, Theorem B.1), if $E>0$ is big enough (depending on $\eps>0$, of course), we can ensure the existence of a non-trivial integer solution $(n_i)_{0\leq i\leq d-1}$ to the latter system. For such a solution, set $\tilde{\eta} = \sum_{i=0}^{d-1} n_i \alpha^i$ and $\eta = \{\tilde{\eta}\}$ (the fractional part), which satisfies \eqref{SuffCond} and the conclusion of the lemma.
	\end{proof}
	To prove that it is impossible to obtain a uniform H\"older exponent in the Salem case, we will use the next lemma.

	Let $A$ be a primitive matrix with a simple spectrum, having the  Perron-Frobenius eigenvalue $\theta$ and the second in modulus eigenvalue $\theta_2$.
	Let $\bu$ be the  Perron-Frobenius eigenvector of norm 1, so that $A\mathbf{u} = \theta \mathbf{u}$, and let $\Hk_2$ be the (real) $A$-invariant subspace complementary to 
	$\Span\{\bu\}$. We can define a norm $\norm{\cdot}$ in $\R^d$ adjusted to the matrix $A$, namely, 
	$\norm{\mathbf{x}} = \norm{c_1 \mathbf{u} + \mathbf{v}} = \max(\abs{c_1}, \norm{\mathbf{v}}_{\mathcal{H}_2})$, where $\bv\in \Hk_2$ and
	$\norm{\mathbf{v}}_{\mathcal{H}_2}$ is a norm in $\mathcal{H}_2$ such that
	\begin{equation}
		\norm{A\mathbf{v}}_{\mathcal{H}_2} \leq \abs{\theta_2}\cdot \norm{\mathbf{v}}_{\mathcal{H}_2}.
	\end{equation}
	
	\begin{lemma}\label{lem:iter2}
		Let $A$ be a matrix as above, and let $\eps > 0$ be such that $\abs{\theta_2} + \eps < \theta$ and $\delta \in (0, 1)$. Let $\Delta A$ be a (complex) matrix such that
		\begin{equation} \label{H11}
			\norm{\Delta A} < \min\{\eps, \delta(\theta - \abs{\theta_2} - \eps) \}.
		\end{equation} 
		Then for any $\mathbf{x} = c_1\mathbf{u} + \mathbf{v}$, with $\mathbf{v}\in\mathcal{H}_2$ and $\norm{\mathbf{v}}_{\mathcal{H}_2} \leq \delta c_1$ we have
		\begin{equation}\label{C11}
			(A + \Delta A)\mathbf{x} = c_2\mathbf{u} + \mathbf{w}, \text{ with } c_2>c_1(\theta-\eps),\, \mathbf{w}\in\mathcal{H}_2,\, \norm{\mathbf{w}}_{\mathcal{H}_2} < c_2\delta.
		\end{equation}
	\end{lemma}
	
	\begin{proof}
		We have
		\[
		A\mathbf{x} = c_1 \theta \mathbf{u} + A\mathbf{v},\quad A\mathbf{v} \in \mathcal{H}_2 ,\,\quad \norm{A\mathbf{v}}_{\mathcal{H}_2} \leq\abs{\theta_2} \cdot  \norm{\mathbf{v}}_{\mathcal{H}_2}.
		\]
		Denote by $\text{Proj}_\mathbf{u}$ the projection onto $\mathbf{u}$ parallel to $\mathcal{H}_2$ and by $\text{Proj}_2$ the projection onto $\mathcal{H}_2$ parallel
		to $\mathbf{u}$. Observe that using our adjusted norm we have $\norm{\text{Proj}_\mathbf{u}} = \norm{\text{Proj}_2} = 1$. Now,
		\[
		c_2 \mathbf{u} = \text{Proj}_\mathbf{u} ((A + \Delta A) \mathbf{x}) = c_1 \theta\mathbf{u}  + \text{Proj}_\mathbf{u} (\Delta A \mathbf{x}),
		\]
		hence
		\[
		c_2 \geq c_1 \theta - \norm{\text{Proj}_\mathbf{u}} \cdot \norm{\Delta A} \cdot \norm{\mathbf{x}} = c_1 (\theta - \norm{\Delta A}),
		\]
		because by our assumption, $\text{Proj}_2\mathbf{x} = \mathbf{v}$ and $\norm{\mathbf{v}}_{\mathcal{H}_2} \leq \delta c_1$, where $\delta\in(0,1)$,
		hence $\norm{\mathbf{x}}=c_1$.
		Thus from $\norm{\Delta A} < \eps$ in 
		\eqref{H11} we obtain $c_2 > c_1 (\theta - \eps)$ in \eqref{C11}. Further,
		\[
		\mathbf{w} = A\mathbf{v} + \text{Proj}_2 (\Delta A \mathbf{x}),
		\]
		hence
		\begin{align*}
			\norm{\mathbf{w}}_{\mathcal{H}_2} &\leq\abs{\theta_2} \cdot \norm{\mathbf{v}}_{\mathcal{H}_2} + \norm{\text{Proj}_2} \cdot \norm{\Delta A} \cdot \norm{\mathbf{x}}\\
			&\leq \abs{\theta_2}  \cdot c_1 \delta + c_1 \norm{\Delta A}\\
			&\leq c_1 \delta\abs{\theta_2}  + c_1\delta(\theta - \abs{\theta_2} - \eps)\\
			&\leq c_1 \delta(\theta - \eps)\\ 
			&\leq c_2 \delta
		\end{align*}
		as desired, where we used \eqref{H11} in the third line.
	\end{proof}

	To finish the proof of Theorem~\ref{th:lower}, 
	we will iterate Lemma \ref{lem:iter2} starting with $\mathbf{x} = \mathbf{u}$. To see why this is enough, we recall that a  formula for the lower local dimension of spectral measures is given by the result below. First we recall the definition of the \textit{spectral cocycle} defined in \cite{BuSo20}.
	
	\begin{defi}\label{generalCocycle}
		Let $\zeta(a) = w_1\dots w_{k_a}$ and $e(x) = e^{-2\pi ix}$. The \textit{spectral cocycle} is the complex matrix cocycle over the toral endomorphism $\xi \mapsto \Sf_\zeta^{\sf T}$ on $\mathbb{T}^d = (\R/\Z)^d$, with the $(a,b)\in\A^2$ entry given by
		\begin{align*}
			\mathscr{C}_{\zeta}(\xi,1)(a,b) &= \sum_{j=1}^{k_a} \delta_{w_j,b} \:e(\xi_{w_1} + \dots + \xi_{w_{j-1}}), \quad \xi \in \mathbb{T}^d; \\
			\mathscr{C}_{\zeta}(\xi,n) &=  \mathscr{C}_{\zeta}(({\sf S}^{\sf T}_{\zeta})^{n-1}\xi,1) \dots \mathscr{C}_{\zeta}(\xi,1).
		\end{align*}
		The local upper Lyapunov
		exponent of the cocycle at the point $\xi\in \T^d$ is defined by
		$$
		\chi^+(\xi) = \chi^+_{\zeta}(\xi):= \limsup_{n\to \infty} \frac{1}{n} \log \|\mathscr{C}_\zeta(\xi,n)\|.
		$$
	\end{defi}
	
	\begin{sloppypar}
		\begin{remark}\label{rm:SpecCo}
			Note that for $\mathbf{0} = (0,\dots,0)^{\sf T}$ we have $\mathscr{C}_{\zeta}(\mathbf{0},1) = {\sf S}^{\sf T}_\zeta$ and 
			$\mathscr{C}_{\zeta}(\om\mathbf{1},n) = \mathscr{C}_{\zeta}\bigl(\om( {\sf S}^{\sf T}_{\zeta})^{n-1}\mathbf{1},1\bigr)$.
		\end{remark}
	\end{sloppypar}

	
	\begin{prop}
		Let $\zeta$ be a primitive aperiodic substitution on $\Ak$, and let $\sig_f$ be the spectral measure on $[0,1)\cong \T$, corresponding to 
		$f =  \sum_{j=0}^{d-1} b_j \One_{[j]}$, a cylindrical function of level 0 on $X_\zeta$. Then for any $\om \in (0,1)$, such that $\chi^+(\om\mathbf{1})>0$, 
		for Lebesgue-a.e.\ $(b_0,\ldots,b_{d-1})$ holds
		\be \label{eq:Lyap-dim}
		\und{d}(\sig_f,\om) = 2 - \frac{2\chi^+(\om\mathbf{1})}{\log \theta}\,,
		\ee
		where $\theta$ is the PF eigenvalue of the substitution matrix $\Sf_\zeta$.
	\end{prop}

	\begin{proof} 
		Consider the suspension flow with a constant-one roof function over the
		substitution automorphism, acting on $X_\zeta\times [0,1]$, and consider $F(x,t) = f(x)$, a simple cylindrical function on this space. For the suspension $\R$-action 
		there is a spectral measure  on $\R$, corresponding to $F$, which we denote by $\sig_F$.
		By \cite[Lemma 5.6]{BerSol}, there is a simple relation between $\sig_F$ and the spectral measure $\sig_f$ on $\T\cong [0,1)$:
		\[d\sig_F(\om)=\left(\frac{\sin(\pi \om)}{\pi \om}\right)^2 \cdot d\sig_f(\om),\ \ \om\in (0,1),\]
		extended 1-periodically to $\R\setminus \Z$. Corollary 4.4(i) in \cite{BuSo20}, restricted to the special case of a single substitution and the constant-one roof function, implies that for any $\om \in \R\setminus \Z$, for Lebesgue-a.e.\ $(b_0,\ldots,b_{d-1})$, such that $\chi^+(\om\mathbf{1})>0$, holds
		$$
		\und{d}(\sig_F,\om) = 2 - \frac{2\chi^+(\om\mathbf{1})}{\log \theta}\,,
		$$
		and \eqref{eq:Lyap-dim} follows.
	\end{proof}

	Let us say that a cylindrical function $f =  \sum_{j=0}^{d-1} b_j \One_{[j]}$ is {\em admissible} if \eqref{eq:Lyap-dim} holds for all $\om\in \Z[\alpha]\setminus \{0\}$. 
	Clearly, a.e.\ $f\in \Cyl(X_\zeta)$ is
	admissible. Moreover, since for any constant $c\in\C$, $\sig_{f+c\One}$ is absolutely continuous with respect to $\sig_f + \sig_{\One}=\sig_f + \delta_0$, see \cite[Corollary 2.1]{Queff}, we have
	$$
	\und{d}(\sig_{f+c\One},\om)= \und{d}(\sig_f,\om)\ \ \mbox{for all}\ \om\in (0,1),
	$$
	and hence a.e.\ $f\in \Cyl(X_\zeta)$ of mean zero is admissible as well.

	\begin{proof}[Proof of Theorem~\ref{th:lower}]
		(i)  Let $f$ be an admissible cylindrical function.
		Suppose by contradiction there exists a uniform $\vartheta>0$ such that for all $\om \in (0,1)$ we have 
		\[
		\sigma_{f}(B_r(\omega)) \leq Cr^{\vartheta}, \quad \forall r<r_0.
		\]
		In other words, for all $\om \in (0,1)$ we have $\underline{d}(\sig_f,\omega) \geq \vartheta$. The matrix-function $\xi \mapsto \mathscr{C}_{\zeta}(\xi,1)$ is uniformly continuous on the torus $\T^d$, hence, in view of Remark \ref{rm:SpecCo}, for any $\eps>0$ there exists $\om=\eta$ from Lemma \ref{SalemPropMatrix} for an appropriate $\eps'=\eps'(\eps)$ (using the continuity of the matrix-function), such that
		\[
		\norm{ {\sf S}_\zeta^{\sf T} - \mathscr{C}_{\zeta}(\om ( {\sf S}^{\sf T}_{\zeta})^{n-1} \mathbf{1},1)} < \eps, \quad \text{ for all } n\geq 1.
		\]
		This means  we can apply Lemma \ref{lem:iter2} iteratively, with $A = {\sf S}_\zeta^{{\sf T}}$, $\mathbf{x} = \mathbf{u}$, the Perron-Frobenius eigenvector, and $\Delta A = {\sf S}_\zeta^{{\sf T}} - \mathscr{C}_{\zeta}\bigl(\om{\sf S}^{\sf T}_{\zeta})^{n-1}\mathbf{1},1\bigr)$ at the $n$-th step, for $n=1,\ldots,N$.	After $N$ iterations of the lemma, we obtain $\mathscr{C}_{\zeta}(\om\mathbf{1},N) = c_N \bu + \bw$, with $\bw\in \Hk_2$ and $c_N \ge (\theta-\eps)^N$, hence
		\[
		\norm{\mathscr{C}_{\zeta}(\om\mathbf{1},N) } \geq \const\cdot(\theta-\eps)^N,
		\]
		and we obtain
		\[
		\chi^+(\omega\mathbf{1}) \geq \log(\theta-\eps).
		\]
		But this yields $2 - \dfrac{2\log(\theta-\eps)}{\log\theta}\geq  2 - \dfrac{2\chi^+(\omega\mathbf{1})}{\log\theta} = \underline{d}(\sig_f,\omega) \geq \vartheta > 0$, 
		in view of \eqref{eq:Lyap-dim}, which leads to a contradiction for $\eps>0$ small enough.
		
		\medskip

		(ii) This follows from part (i) immediately, by \cite[Theorem 3.7]{Knill},
		which is a general version of Last's theorem
		\cite[Theorem 3.1]{Last96}.
	\end{proof}
	

\end{document}